\documentclass[reqno,11pt]{amsart}
\usepackage{amssymb,color,hyperref,mathrsfs,stmaryrd}

\usepackage[usenames,dvipsnames]{xcolor}
\usepackage[color,curve,matrix,arrow]{xy}

\title[Products in Fusion Systems]{
Products in Fusion Systems}

\author[E.~Henke]{Ellen Henke}
\address{Department of Mathematical Sciences, 
University of Copenhagen, Universitetsparken 5, 
DK--2100 Copenhagen, Denmark}
\thanks{The author was supported by the Danish National Research
  Foundation (DNRF) through the Centre for Symmetry and
  Deformation.}

\begin{document}

\newtheorem{notation}{Notation}[section]
\newtheorem{definition}[notation]{Definition}
\newtheorem{satz}[notation]{}
\newtheorem{lemma}[notation]{Lemma}
\newtheorem{corollary}[notation]{Corollary}
\newtheorem{theorem}[notation]{Theorem}
\newtheorem{proposition}[notation]{Proposition}
\newtheorem{hyp}[notation]{Hypothesis}
\newtheorem{remark}[notation]{Remark}
\newtheorem{mdefinition}{Definition}
\newtheorem{mtheorem}{Theorem}
\newtheorem{mnotation}{Notation}
\newtheorem{mprop}{Proposition}
\newtheorem{mhyp}{Hypothesis}
\numberwithin{equation}{notation}

\theoremstyle{definition}
\newtheorem{ex}[notation]{Example}

\def \<{\langle }
\def \>{\rangle }
\newcommand{\sy}{\textrm{Syl}}
\newcommand{\ov}{\overline}
\newcommand{\m}{\mathcal}
\newcommand{\wt}{\widetilde}
\newcommand{\wh}{\widehat}
\newcommand{\F}{\mathcal{F}}
\newcommand{\G}{\mathcal{G}}
\newcommand{\ch}{\;char_{\mathcal{F}}\;}
\newcommand{\Ac}{\operatorname{A}^\circ}
\newcommand{\A}{\operatorname{A}}
\newcommand{\Vc}{V^\circ}
\newcommand{\FN}{\F_\m{N}}
\newcommand{\FNT}{\F_\m{N}^*}
\newcommand{\Q}{\mathcal{Q}}
\newcommand{\subn}{\unlhd\unlhd\;}
\newcommand{\vphi}{\varphi}
\newcommand{\Hom}{\operatorname{Hom}}
\newcommand{\Aut}{\operatorname{Aut}}
\newcommand{\Inn}{\operatorname{Inn}}
\newcommand{\Iso}{\operatorname{Iso}}
\newcommand{\Syl}{\operatorname{Syl}}
\newcommand{\id}{\operatorname{Id}}
\newcommand{\N}{\operatorname{N}}
\newcommand{\C}{\operatorname{C}}
\newcommand{\mynote}[1]{{\bf [ {#1}]}}

\begin{abstract}
We revisit the notion of a product of a normal subsystem with a $p$-subgroup as defined by Aschbacher \cite[Chapter 8]{AschbacherGeneralized}. In particular, we give a previously unknown, more transparent construction. 
\end{abstract}

\maketitle

\section{Introduction}

\textit{Saturated fusion systems} are categories mimicking important properties of fusion in finite groups. They were (under a different name) first defined and studied by Puig in the early 1990's, mostly for the purposes of block theory; see \cite{Puig2006} and \cite{PuigBookFrobenius}. Later, Broto, Levi and Oliver introduced in \cite{BrotoLeviOliver2003} the now standard notation and terminology. They also extended Puig's theory for the study of classifying spaces of finite groups.

\smallskip

From the very beginning, translating group theoretical concepts into the framework of fusion systems played a vital role in developing the theory from an algebraic point of view. Already Puig has introduced normalizers and centralizers of $p$-subgroups in fusion systems, normal and central subgroups, factor systems, and a notion of normal subsystems. More recently, in two fundamental papers \cite{AschbacherNormal,AschbacherGeneralized}, Aschbacher has built up an increasingly rich theory. His main motivation was to provide a framework in which portions of the classification of finite simple groups can be carried out in the category of fusion systems, hopefully leading to a simpler proof. 

\smallskip

Even though concepts borrowed from finite group theory became fundamental for the understanding of fusion systems, many constructions which are elementary in groups are difficult or perhaps even impossible in fusion systems. For example, if $N$ is a normal subgroup of a group $G$ then, for any subgroup $H$ of $G$, the product $NH$ is trivially again a subgroup of $G$. If we, in contrast, consider a saturated fusion system, products of normal subsystems with other saturated subsystems are so far only constructed in very special cases. Aschbacher \cite[Thm.~3]{AschbacherGeneralized} has proved the existence of a product of two normal subsystems provided their underlying $p$-groups commute. Moreover, he has defined a product of a normal subsystem with a $p$-subgroup; see Theorem~5 and Chapter~8 in \cite{AschbacherGeneralized}. In this paper we aim to review the latter concept. The reason is firstly that, even though Aschbacher's proof is constructive, the explicit description of the product system is quite complicated, so we would like to give an easier construction. Secondly, we seek to simplify parts of the arguments in the proof of \cite[Thm.~5]{AschbacherGeneralized} and to give a more transparent proof. Our proof, like Aschbacher's, uses the existence of models for constrained fusion systems as proved in \cite{BCGLO2005}, and thus relies indirectly on the vanishing of certain higher limits of functors; see also \cite[Section~III.5.2]{Aschbacher/Kessar/Oliver:2011a}. Apart from that our proof is elementary and essentially self-contained. In particular, we avoid the counting argument in \cite[8.1]{AschbacherGeneralized} which relies on the existence of a certain $(S,S)$-biset from  \cite[Prop.~5.5]{BrotoLeviOliver2003} via \cite[Prop.~1]{BCGLO2007}.
This simplification is mainly achieved by exploiting the existence of \textit{well-placed} subgroups which we define in \ref{WellPlacedDef}. However, part of our proof still follows Aschbacher's work.

\smallskip

For the remainder of this paper, we assume the following hypothesis:

\begin{mhyp}\label{MainHyp}
Throughout, $p$ is a prime and $\F$ is a saturated fusion system on a finite $p$-group $S$.  Let $\F_0$ be a normal subsystem of $\F$ on a subgroup $S_0$ of $S$. Let $T$ be a subgroup of $S$ containing $S_0$.
\end{mhyp}

We refer the reader to \cite{Aschbacher/Kessar/Oliver:2011a} for the main definitions regarding saturated fusion systems and normal subsystems. Next we will construct the product $\F_0T$, which we sometimes also denote by $(\F_0T)_\F$ to stress that we form the product inside the given fusion system $\F$. Note that the following definition trivially leads to a notion of the product of $\F_0$ with an arbitrary subgroup $R$ of $S$ just by setting $\F_0R:=\F_0(S_0R)$.

\begin{mdefinition}
For a subgroup $P\leq S$ set 
$$\Ac(P):=\Ac_{\F,\F_0}(P):=\<\vphi\in \Aut_\F(P):\vphi\mbox{ $p^\prime$-element, }[P,\vphi]\leq P\cap S_0\mbox{ and }\vphi|_{P\cap S_0}\in \Aut_{\F_0}(P\cap S_0)\>.$$
The \textbf{product} of $\F_0$ with $T$ in $\F$ is the fusion system 
$$\F_0T:=(\F_0T)_\F:=\<\Ac(P):P\leq T\mbox{ and }P\cap S_0\in \F_0^c\>_T.$$
\end{mdefinition}

Here, for any set $\m{H}$ consisting of $\F$-morphisms between subgroups of $T$, we write $\<\m{H}\>_T$ for the smallest subsystem of $\F$ on $T$ containing every element of $\m{H}$.

\smallskip

In the definition above, it might at first seem artificial to restrict attention to the subgroups $P$ of $T$ with $P\cap S_0\in\F_0^c$. However, this is indeed essential. We prove in \ref{Ac} that  $\Ac(P)=O^p(\Aut_{\F_0T}(P))$ for any $P\leq T$ with $P\cap S_0\in \F_0^c$. In contrast, for an arbitrary subgroup $P$ of $T$, $\Ac(P)$ does not need to be contained in $\Aut_{\F_0T}(P)$ as we show in Example~\ref{7.5}. Thus, it seems that there is no easy way of describing $\Aut_{\F_0T}(P)$.  Nevertheless, according to the theorem we state next, the subsystem $\F_0T$ is in fact the only saturated subsystem of $\F$ which can sensibly play the role of a product of $\F_0$ with $T$.

\begin{mtheorem}\label{MainThm}
The fusion system $\F_0T$ is a saturated subsystem of $\F$ on $T$. Furthermore, $\F_0T$ is the unique saturated subsystem $\m{E}$ of $\F$ on $T$ with $O^p(\m{E})=O^p(\F_0)$. 
\end{mtheorem}

The above theorem is essentially \cite[Thm.~5]{AschbacherGeneralized} except for the concrete description of $\F_0T$. The uniqueness implies in particular that our subsystem $\F_0T$ coincides with the subsystem $\F_0T$ defined by Aschbacher. For the uniqueness statement it is actually important to form the product ``internally'', i.e. inside of a fixed fusion system $\F$; see Example~\ref{7.4}.

\smallskip

If $G$ is a finite group, $S\in\Syl_p(G)$ and $N\unlhd G$, then by \cite[Prop.~I.6.2]{Aschbacher/Kessar/Oliver:2011a}, $\F_{S\cap N}(N)$ is a normal subsystem of $\F_S(G)$. As stated in the next proposition, the fusion system product coincides, in the group case, with the fusion system of the usual product of subgroups.

\begin{mprop}\label{GroupProp}
Suppose $\F=\F_S(G)$ for some finite group $G$ with $S\in \Syl_p(G)$, and there exists a normal subgroup $N$ of $G$ such that $S_0=S\cap N$ and $\F_0=\F_{S_0}(N)$.  Then $\F_0T=\F_T(NT)$.
\end{mprop}

By the Hyperfocal Subgroup Theorem of Puig \cite[\S 1.1]{Puig2000h} and \cite[Thm.~7.4]{Aschbacher/Kessar/Oliver:2011a}, 
$$O^p(\F_S(G))=\F_{S\cap O^p(G)}(O^p(G))$$ 
for any finite group $G$ with $S\in Syl_p(G)$. Thus, under the hypothesis of Proposition~\ref{GroupProp}, $O^p(\F_0)=\F_{S_0\cap O^p(N)}(O^p(N))=\F_{T\cap O^p(NT)}(O^p(NT))=O^p(\F_T(NT))$ as $O^p(NT)=O^p(N)$. Thus, Proposition~\ref{GroupProp} could be obtained as a consequence of Theorem~\ref{MainThm}. However, we need to prove Proposition~\ref{GroupProp} first, because it is applied in the proof of Theorem~\ref{MainThm} to constrained local subsystems, which by \cite{BCGLO2005} come from a finite group.

\smallskip

The overall structure of this paper is as follows: After some  preliminary results in Section~\ref{S2}, Proposition~\ref{GroupProp} is proved in Section~\ref{S3}. In Section~\ref{S4} we prove various properties of $\F_0T$, which in Sections~\ref{S5} and \ref{S6} are used to prove Theorem~\ref{MainThm}. We conclude in Section~\ref{S7} with some final remarks and examples. In particular, we explore in Subsection~\ref{factor} connections to factor systems. We adapt the definitions and notations from \cite{Aschbacher/Kessar/Oliver:2011a}, especially the ones from Part~II, as we write our functions on the right side. Furthermore, throughout this paper, we use the following notation:

\begin{mnotation}
 Set $\m{D}:=\F_0T$ and, for any $P\leq T$, $P_0:=P\cap S_0$.
\end{mnotation}

\subsubsection*{Acknowledgment} The author would like to thank Prof. Michael Aschbacher for many helpful and stimulating discussions and for hosting her for four weeks at Caltech in November and December 2011.

\section{Preliminaries}\label{S2}

In this section we collect some lemmas regarding fusion systems, which are necessary later on. According to Hypothesis~\ref{MainHyp}, $\F$ is a saturated fusion system on $S$. So in addition to the weak axioms \cite[Def.~2.1]{Aschbacher/Kessar/Oliver:2011a} that are satisfied in any fusion system, two non-trivial axioms need to be satisfied, the \textit{Sylow axiom} and the \textit{extension axiom}; see \cite[Prop.~2.5]{Aschbacher/Kessar/Oliver:2011a} and also \cite[Def.~2.2]{Aschbacher/Kessar/Oliver:2011a} for an equivalent definition. The extension axiom says that, for subgroups $P,Q\leq S$ with $Q$ fully $\F$-centralized, each $\vphi\in\Iso_\F(P,Q)$ extends to an element of $\Hom_\F(N_\vphi,S)$, where
$$N_\vphi:=N_\vphi^\F:=\{g\in \N_S(P):(c_g|_P)\vphi^*\in\Aut_S(Q)\}.$$
By the next remark, this is actually a natural condition, since $N_\vphi$ is the largest subgroup of $\N_S(P)$ to which $\vphi$ can possibly be extended.

\begin{remark}\label{littleLemma}
  Let $P\unlhd X\leq S$, and let $\psi:X\rightarrow S$ be a group monomorphism (not necessarily in $\F$)  such that $\vphi:=\psi|_P\in \Hom_\F(P,P\psi)$. Then for all $g\in X$, $(c_g|_P)\vphi^*=c_{g\psi}|_{P\psi}$. In particular, $X\leq N_\vphi$ and $\Aut_X(P)\vphi^*=\Aut_{X\psi}(P\psi)$.
\end{remark}

\begin{proof}
 For $h\in P\psi$, $h((c_g|_P)\vphi^*)=((h\psi^{-1})^g)\psi=h^{g\psi}=h(c_{g\psi}|_{P\psi})$. 
\end{proof}

As it will become apparent in the proofs, the above remark has also some very practical consequences, since in many cases it allows to extend a morphism in a subsystems of $\F$, provided there exists an extension in $\F$. In this connection also the next remark is useful. Recall that, given a (not necessarily saturated) fusion system $\m{E}$ on a finite $p$-group $R$, a subgroup $Q$ of $R$ is called \textit{fully automized} in $\m{E}$ if $Aut_R(Q)\in \Syl_p(Aut_\m{E}(Q))$.

\begin{remark}\label{NphicapR}
Suppose $\m{E}$ is a subsystem of $\F$ on a subgroup $R$ of $S$. Let $P\leq R$ and $\vphi\in \Hom_\m{E}(P,R)$ such that $P\vphi$ is fully automized in $\m{E}$. Then $\Aut_{N_\vphi\cap R}(P)\vphi^*\leq \Aut_R(P)$ and
$$N_\vphi^\m{E}=N_\vphi\cap R.$$
\end{remark}

\begin{proof}
Note $\Aut_R(P\vphi)\leq \Aut_S(P\vphi)\cap \Aut_\m{E}(P\vphi)$, so as $P\vphi$ is fully automized in $\m{E}$, $\Aut_R(P\vphi)=\Aut_S(P\vphi)\cap \Aut_\m{E}(P\vphi)$. Then by definition of $N_\vphi$, $\Aut_{N_\vphi\cap R}(P)\vphi^*\leq Aut_S(P\vphi)\cap Aut_\m{E}(P\vphi)=\Aut_R(P\vphi)$ which yields the assertion.
\end{proof}

The next rather specialized result gives a connection between two potentially different extensions of a morphism.

\begin{lemma}\label{CommInC(P0)}
Let $P\in\F$, $Q\unlhd P$, $\gamma\in \Aut_\F(P)$ and $\beta\in \Hom_\F(P,S)$ such that $\beta|_Q=\gamma|_Q$. Then $[\C_P(\gamma),\beta]\leq \C_S(Q\beta)$.
\end{lemma}

\begin{proof}
Observe first that $Q\beta=Q\gamma$ is normal in $P\gamma=P$. Let $x\in \C_P(\gamma)$. Using \ref{littleLemma} we obtain
$$c_{x\beta}|_{Q\beta}=c_x|_{Q}(\beta|_{Q})^*=c_x|_{Q}(\gamma|_{Q})^*=c_{x\gamma}|_{Q\gamma}=c_x|_{Q\gamma}=c_x|_{Q\beta}.$$
Hence
$$c_{x^{-1}(x\beta)}|_{Q\beta}=(c_x|_{Q\beta})^{-1}c_{x\beta}|_{Q\beta}=(c_x|_{Q\beta})^{-1}c_x|_{Q\beta}=\id_{Q\beta}.$$
This implies $[x,\beta]=x^{-1}(x\beta)\in \C_S(Q\beta)$ and thus the assertion.
\end{proof}

We now turn attention to the normal subsystem $\F_0$ of $\F$; see \cite[Section~I.6]{Aschbacher/Kessar/Oliver:2011a} for a detailed introduction to normal subsystems. The next two lemmas are concerned with properties of subgroups of $S_0$.

\begin{lemma}\label{PrelI}
\begin{itemize}
\item[(a)] For any $P_0\in \F_0^c$, $P_0^\F\subseteq \F_0^c$.
\item[(b)] Let $P_0\in \F_0^f$ and $\alpha\in \Hom_\F(\N_{S_0}(P_0),S_0)$. Then $P_0\alpha\in \F_0^f$ and $\N_{S_0}(P_0)\alpha=\N_{S_0}(P_0\alpha)$.
\end{itemize}
\end{lemma}

\begin{proof}
Note that every element of $\Aut_\F(S_0)$ induces an automorphism of $\F_0$ and thus maps every $\F_0$-centric subgroup to an $\F_0$-centric subgroup and every fully $\F_0$-normalized subgroup to a fully $\F_0$-normalized subgroup. If $P_0\in\F_0^c$ and $\vphi\in \Hom_\F(P_0,S)$ then by the Frattini argument for fusion systems \cite[Prop.~I.6.4]{Aschbacher/Kessar/Oliver:2011a}, $\vphi=\vphi_0\beta$ for $\vphi_0\in \Hom_{\F_0}(P_0,S_0)$ and some $\beta\in \Aut_\F(S_0)$. Then $P_0\vphi_0\in \F_0^c$ as $P_0\in\F_0^c$. Hence, also $P_0\vphi=(P_0\vphi_0)\beta\in\F_0^c$ proving (a). Let now $P_0$ and $\alpha$ be as in (b). Then again by the Frattini argument \cite[Prop.~I.6.4]{Aschbacher/Kessar/Oliver:2011a}, $\alpha=\alpha_0\beta$ for some $\alpha_0\in\Hom_{\F_0}(\N_{S_0}(P_0),S_0)$ and some $\beta\in\Aut_\F(S_0)$. As $P_0\in\F_0^f$ and $\N_{S_0}(P_0)\alpha_0\leq \N_{S_0}(P_0\alpha_0)$, we have $P_0\alpha_0\in \F_0^f$ and $\N_{S_0}(P_0)\alpha_0= \N_{S_0}(P_0\alpha_0)$. Hence, $P_0\alpha=P_0\alpha_0\beta\in\F_0^f$ and $\N_{S_0}(P_0)\alpha=\N_{S_0}(P_0\alpha_0)\beta=\N_{S_0}(P_0\alpha)$, which proves (b).
\end{proof}

\begin{lemma}\label{PrelII}
 Let $Q_0\leq S_0$ such that $Q_0\in\F^f$. Then $Q_0\in\F_0^f$.
\end{lemma}

\begin{proof}
 Let $P_0\in Q_0^{\F_0}\cap\F_0^f$. As $Q_0\in\F^f$, it follows from  \cite[Lemma~II.3.1]{Aschbacher/Kessar/Oliver:2011a} that there exists $\vphi\in\Hom_\F(\N_S(P_0),S)$ such that $P_0\vphi=Q_0$. Then $\N_{S_0}(P_0)\vphi\leq \N_{S_0}(Q_0)$, so $Q_0\in\F_0^f$ as $P_0\in\F_0^f$.
\end{proof}

We conclude this section with a technical result needed in the proof of \ref{P0Extend}. It gives some properties of extensions of morphisms between subgroups of $S_0$.

\begin{lemma}\label{NewLemma}
 Let $\m{E}$ be a subsystem on $T$, $V_0\in\F_0^c$, $P_0\in V_0^\F$, $\alpha\in\Hom_\m{E}(P_0,V_0)$, $Q_0=\N_\alpha^{\m{E}}\cap S_0$ and $\hat{\alpha}\in\Hom_{\m{E}}(Q_0,S_0)$ such that $\hat{\alpha}|_{P_0}=\alpha$. Then 
$$\Aut_{\N_\alpha^{\m{E}}}(Q_0)\hat{\alpha}^*\leq \Aut_{\N_T(V_0)}(Q_0\hat{\alpha})\C_{\Aut_\m{E}(Q_0\hat{\alpha})}(V_0)$$
and $\{t\in \N_T(V_0):c_t|_{V_0}\in \Aut_{\N_\alpha^{\m{E}}}(P_0)\alpha^*\}\subseteq \N_T(Q_0\hat{\alpha})$.
\end{lemma}

\begin{proof}
 Set $W_0:=Q_0\hat{\alpha}$ and let $t\in \N_T(V_0)$ such that $c_t|_{V_0}\in \Aut_{\N_\alpha^\m{E}}(P_0)\alpha^*$. Observe $Q_0\unlhd \N_\alpha^\m{E}$ and thus $\Aut_{Q_0}(P_0)\unlhd \Aut_{\N_\alpha^\m{E}}(P_0)$. Using \ref{littleLemma}, we get $\Aut_{W_0}(V_0)=\Aut_{Q_0}(P_0)\alpha^*\unlhd \Aut_{\N_\alpha^\m{E}}(P_0)\alpha^*$. In particular, $\Aut_{W_0}(V_0)$ is normalized by $c_t|_{V_0}$ and thus, again by \ref{littleLemma}, 
$$\Aut_{W_0^t}(V_0)=\Aut_{W_0}(V_0)(c_t|_{V_0})^*=\Aut_{W_0}(V_0).$$
Hence, $W_0^t\leq W_0\C_{S_0}(V_0)=W_0$ as $V_0\in\F_0^c$. This proves $t\in \N_T(W_0)$ and thus the second part of the assertion. For the first part let $\psi\in \Aut_{\N_\alpha^\m{E}}(Q_0)\hat{\alpha}^*$ and note that $\psi|_{V_0}\in \Aut_{\N_\alpha^\m{E}}(P_0)\alpha^*\leq \Aut_T(V_0)$. Hence, there exists $s\in \N_T(V_0)$ such that $\psi|_{V_0}=c_s|_{V_0}$. By what we have proved before, $s\in \N_T(W_0)$, so $\psi (c_s|_{W_0})^{-1}\in \C_{\Aut_{\m{E}}(W_0)}(V_0)$. This completes the proof.
\end{proof}

\section{The proof of Proposition~\ref{GroupProp}}\label{S3}

We prove the following slightly stronger version of Proposition~\ref{GroupProp}:

\begin{proposition}\label{Prop1Str}
Suppose $\F=\F_S(G)$ for some finite group $G$ with $S\in \Syl_p(G)$, and there exists a normal subgroup $N$ of $G$ such that $S_0=S\cap N$ and $\F_0=\F_{S_0}(N)$.  Then $\F_0T=\F_T(NT)$ and $\Ac(P)=O^p(\Aut_{\F_0T}(P))=O^p(\Aut_N(P))$ for any $P\leq T$ with $P\cap S_0\in\F_0^c$.
\end{proposition}

\begin{proof}
Observe that, for any $P\leq T$, $\N_{NT}(P)/\N_N(P)\cong \N_{NT}(P)N/N\leq TN/N$ is a $p$-group and hence $O^p(\N_{NT}(P))=O^p(\N_N(P))$. This implies
\begin{equation}\label{Group1}
O^p(\Aut_{NT}(P))=O^p(\Aut_N(P))\leq \Ac(P)\mbox{ for any }P\leq T.
\end{equation}
 If $P\in\F_T(NT)^{frc}$, then by \cite[7.18]{AschbacherGeneralized}, $P_0\in \F_0^c$. Moreover, $\Aut_T(P)\in \Syl_p(\Aut_{NT}(P))$ and thus $\Aut_{NT}(P)=O^p(\Aut_{NT}(P))\Aut_T(P)$. Hence, by definition of $\F_0T$, (\ref{Group1}) and Alperin's Fusion Theorem \cite[Thm.~A.10]{BrotoLeviOliver2003}, we have $\F_T(NT)\subseteq \F_0T$.

\smallskip

To prove $\F_0T\subseteq \F_T(NT)$ let $P\leq T$ such that $P_0\in\F_0^c$. We need to show that $\Ac(P)\leq \Aut_{NT}(P)$. Let $\vphi\in \Aut_\F(P)$ be a $p^\prime$-element such that $[P,\vphi]\leq P_0$ and $\vphi|_{P_0}\in \Aut_{\F_0}(P_0)$. Then there exist $p^\prime$-elements $g\in \N_G(P)$ and $n\in \N_N(P_0)$ such that $\vphi=c_g|_P$ and $\vphi|_{P_0}=c_n|_{P_0}$. Then $gn^{-1}\in \C_G(P_0)$. Moreover, $[P,g]\leq P_0$ and thus, by a property of coprime action \cite[8.2.7(a)]{KurzweilStellmacher2004}, $P=P_0\C_P(g)$. Observe that $[\C_P(g),gn^{-1}]=[\C_P(g),n^{-1}]\leq N$. Thus, $[P,gn^{-1}]\leq N$. As $gn^{-1}\in \C_G(P_0)$ and $\C_G(P_0)$ is normalized by $P$, this implies $[P,gn^{-1}]\leq \C_N(P_0)$. Set
$$U:=O_{p^\prime}(\C_N(P_0)),\;X:=\N_G(PU),\;C:=\C_X(PU/P_0U)\cap \C_X(P_0).$$
Observe that $U=O_{p^\prime}(PU)\unlhd X$ and thus $Y:=\C_X(PU/U)\unlhd X$. Set 
$$\ov{X}:=X/Y.$$
Since $P_0\in\F_0^c$, by \cite[Lemma~A.4]{BrotoLeviOliver2003b}, $\C_N(P_0)=Z(P_0)\times U$. From above, $gn^{-1}\in \C_G(P_0)\leq N_G(U)$ and $[P,gn^{-1}]\in \C_N(P_0)=Z(P_0)U$, so $gn^{-1}\in X$ and $gn^{-1}\in C$. Since $P_0=P\cap N\unlhd \N_G(P)$, it follows $\N_G(P)\leq \N_G(U)$ and hence $g\in \N_G(P)\leq X$. Thus we have shown that $g,n\in X$ and $gC=nC$, whence also  $\ov{g}\ov{C}=\ov{n}\ov{C}$. Observe that $\ov{C}$ is a $p$-group by \cite[8.2.2(b)]{KurzweilStellmacher2004}. It follows that $\<\ov{g}\>\ov{C}=\<\ov{n}\>\ov{C}$ and $\<\ov{g}\>,\<\ov{n}\>$ are $p^\prime$-Hall subgroups of $\<\ov{g}\>\ov{C}$. Thus, as $\<\ov{g}\>\ov{C}$ is solvable, $\<\ov{g}\>$ and $\<\ov{n}\>$ are conjugate in $\<\ov{g}\>\ov{C}$. This implies $\<\ov{g}\>\leq \<\<\ov{n}\>^{\ov{X}}\>\leq \ov{N\cap X}$ and thus $\ov{g}\in \ov{N\cap X}$. Therefore, there exists $\tilde{n}\in N$ such that $g\tilde{n}^{-1}\in Y$. Observe that $[\N_Y(P),P]\leq P\cap [Y,P]\leq P\cap U=1$ and so $\N_Y(P)\leq \C_G(P)$. Hence, by a Frattini argument, $g\tilde{n}^{-1}\in YP=\N_{YP}(P)(PU)=\N_Y(P)PU\leq \C_G(P)PN$. It follows $g\in \C_G(P)PN$ and thus $\vphi=c_g|_P\in \Aut_{PN}(P)\leq \Aut_{TN}(P)$. This proves
$$\Ac(P)\leq O^p(\Aut_{NT}(P)).$$
So $\F_0T\subseteq \F_T(NT)$ and, by (\ref{Group1}), $\Ac(P)=O^p(\Aut_{NT}(P))=O^p(\Aut_N(P))$.
\end{proof}

\section{Properties of $\m{D}=\F_0T$}\label{S4}

\begin{remark}\label{Decompose}
Let $P\leq S_0$ and $\alpha\in \Hom_{\F_0T}(P,S_0)$. Then $\alpha=c_t\alpha_0$ for some $t\in T$ and $\alpha_0\in \Hom_{\F_0}(P^t,S_0)$. Moreover, for any such $t,\alpha_0$, we have $N_\alpha^t\leq N_{\alpha_0}$. 
\end{remark}

\begin{proof}
By construction of $\F_0T$, $\alpha$ is the product of morphisms in $\F_0$ and morphisms induced by $T$. Moreover, for any $Q\leq S_0$, $\beta\in\Hom_{\F_0}(Q,S_0)$ and $s\in T$, we have $\beta (c_s|_{Q\beta})=(c_s|_Q)\hat{\beta}$ where $\hat{\beta}:=(c_s|_Q)^{-1}\beta (c_s|_{Q\beta})\in\F_0$ as $\F_0$ is normal in $\F$. This yields the existence of $t\in T$ and $\alpha_0\in \Hom_{\F_0}(P^t,S_0)$ with $\alpha=c_t\alpha_0$. Using \ref{littleLemma}, we obtain for any such $t,\alpha_0$ that 
$$\Aut_{N_\alpha^t}(P^t)\alpha_0^*=\Aut_{N_\alpha}(P)c_t^*\alpha_0^*=\Aut_{N_\alpha}(P)\alpha^*\leq \Aut_S(P\alpha).$$
Hence, $N_\alpha^t\leq N_{\alpha_0}$.
\end{proof}

\begin{lemma}\label{Lemma0}
Let $P_0\leq S_0$ and $\vphi\in \Hom_{\m{D}}(P_0,S_0)$ such that $P_0\vphi$ is fully $\F_0$-normalized. Then $\vphi$ extends to $\hat{\vphi}\in \Hom_{\m{D}}(N_\vphi\cap S_0,S_0)$.
\end{lemma}

\begin{proof}
By \ref{Decompose}, we have $\vphi=c_t\vphi_0$ for some $t\in T$ and $\vphi_0\in \Hom_{\F_0}(P_0^t,P_0\vphi)$. Moreover, $N_\vphi^t\leq N_{\vphi_0}$. Since $P_0\vphi\in\F_0^f$, $P_0\vphi$ is fully $\F_0$-automized, so by \ref{NphicapR}, $(N_\vphi \cap S_0)^t\leq N_{\vphi_0}^{\F_0}$. Hence, as $\F_0$ is saturated, $\vphi_0$ extends to $\hat{\vphi}_0\in \Hom_{\F_0}((N_\vphi\cap S_0)^t,S_0)$. Thus, $\hat{\vphi}=c_t\hat{\vphi}_0\in \Hom_\m{D}(N_\vphi\cap S_0,S_0)$ extends $\vphi$.  
\end{proof}

\begin{definition}\label{WellPlacedDef}
Let  $P_0\leq S_0$. Set $N_0:=P_0$ and $N_{i+1}:=N_{S_0}(N_i)$ for $i\geq 0$. Then we call $P_0$ \textbf{well-placed} if for all $i\geq 0$, the following conditions hold:
\begin{itemize}
 \item [(i)] $N_i\in \F_0^f$.
 \item [(ii)] $\Aut_T(N_i)\in \Syl_p(\Aut_\m{D}(N_i))$, and
 \item [(iii)] $\N_{\Aut_T(N_{i+1})}(N_i)\in \Syl_p(\N_{\Aut_{\m{D}}(N_{i+1})}(N_i))$.
\end{itemize}
\end{definition}

\begin{lemma}\label{wellplacedExists}
 Let $Q_0\leq S_0$. Then there exists $P_0\in Q_0^\m{D}$ such that $P_0$ is well-placed.
\end{lemma}

\begin{proof}
Let $Q_0$ be a counterexample with $|Q_0|$ maximal. We may assume $Q_0\in \F_0^f$. By construction of $\m{D}$, $\Aut_\m{D}(S_0)=\Aut_{\F_0}(S_0)\Aut_T(S_0)$. As $\F_0$ is saturated, $\Inn(S_0)\in \Syl_p(\Aut_{\F_0}(S_0))$, so $\Aut_T(S_0)\in \Syl_p(\Aut_{\m{D}}(S_0))$ and $S_0$ is well-placed. Hence, $Q_0<S_0$ and thus $Q_0<\N_{S_0}(Q_0)$. Now by maximality of $|Q_0|$, there exists $\vphi\in \Hom_{\m{D}}(\N_{S_0}(Q_0),S_0)$ such that $\N_{S_0}(Q_0)\vphi$ is well-placed. By \ref{PrelI}(b), $Q_0\vphi\in \F_0^f$ and $\N_{S_0}(Q_0)\vphi=\N_{S_0}(Q_0\vphi)$, as $Q_0\in \F_0^f$. Hence, replacing $Q_0$ by $Q_0\vphi$ we may assume that $R_0:=\N_{S_0}(Q_0)$ is well-placed. In particular, $\Aut_T(R_0)\in \Syl_p(\Aut_{\m{D}}(R_0))$ and thus there exists $\psi\in \Aut_{\m{D}}(R_0)$ such that $\Aut_T(R_0)\cap (\N_{\Aut_\m{D}(R_0)}(Q_0)\psi^*)\in \Syl_p(\N_{\Aut_\m{D}(R_0)}(Q_0)\psi^*)$. As $R_0=\N_{S_0}(Q_0)$ and $Q_0\in\F_0^f$, we have $P_0:=Q_0\psi\in\F_0^f$ and $R_0=\N_{S_0}(P_0)$. Then $\N_{\Aut_T(R_0)}(P_0)\in \Syl_p(\N_{\Aut_\m{D}(R_0)}(P_0))$. By \ref{Lemma0}, the elements of $\N_{\Aut_{\m{D}}(P_0)}(\Aut_{S_0}(P_0))$ extend to elements of $\Aut_\m{D}(R_0)$, so
$$\N_{\Aut_{\m{D}}(P_0)}(\Aut_{S_0}(P_0))\cong \N_{\Aut_\m{D}(R_0)}(P_0)/\C_{\Aut_\m{D}(R_0)}(P_0).$$
Hence, as $\N_{\Aut_T(R_0)}(P_0)\in \Syl_p(\N_{\Aut_\m{D}(R_0)}(P_0))$, $\Aut_T(P_0)\in \Syl_p(\N_{\Aut_{\m{D}}(P_0)}(\Aut_{S_0}(P_0)))$. As $P_0\in \F_0^f$ we have $\Aut_{S_0}(P_0)\in \Syl_p(\Aut_{\F_0}(P_0))$. Hence, by the Frattini argument, $\Aut_\m{D}(P_0)=\Aut_{\F_0}(P_0)\N_{\Aut_{\m{D}}(P_0)}(\Aut_{S_0}(P_0))$ and so $\Aut_T(P_0)\in \Syl_p(\Aut_\m{D}(P_0))$. Now $P_0$ is well-placed as $R_0=\N_{S_0}(P_0)$ is well-placed.
\end{proof}

\begin{lemma}\label{Lemma3}
 Let $P_0\in \F_0$ be well-placed. Then $\Aut_\m{D}(P_0)=\Aut_T(P_0)\Aut_{\F_0}(P_0)$. 
\end{lemma}

\begin{proof}
 Let $P_0$ be a counterexample with $|P_0|$ maximal. By construction of $\m{D}$, $\Aut_\m{D}(S_0)=\Aut_T(S_0)\Aut_{\F_0}(S_0)$. Hence, $P_0<S_0$ and thus $P_0<P_1:=\N_{S_0}(P_0)$. As $P_0$ is well-placed, $P_1$ is well-placed. So since $|P_0|$ is maximal, $\Aut_{\m{D}}(P_1)=\Aut_T(P_1)\Aut_{\F_0}(P_1)$. In particular, $O^p(\N_{\Aut_\m{D}(P_1)}(P_0))\leq O^p(\Aut_\m{D}(P_1))\leq \Aut_{\F_0}(P_1)$. As $P_0$ is well-placed, $\N_{\Aut_T(P_1)}(P_0)\in \Syl_p(\N_{\Aut_\m{D}(P_1)}(P_0))$. Hence,
\begin{equation}\label{star}
\N_{\Aut_\m{D}(P_1)}(P_0)=\N_{\Aut_T(P_1)}(P_0)O^p(\N_{\Aut_\m{D}(P_1)}(P_0))=\N_{\Aut_T(P_1)}(P_0)\N_{\Aut_{\F_0}(P_1)}(P_0).
\end{equation}
As $P_0$ is well-placed, $P_0$ is fully $\F_0$-normalized and thus fully $\F_0$-automized. In particular, by the Frattini argument, $\Aut_{\m{D}}(P_0)=\Aut_{\F_0}(P_0)\N_{\Aut_{\m{D}}(P_0)}(\Aut_{S_0}(P_0))$. Hence, it is sufficient to show that $\N_{\Aut_{\m{D}}(P_0)}(\Aut_{S_0}(P_0))\leq \Aut_T(P_0)\Aut_{\F_0}(P_0)$. By \ref{Lemma0}, every element $\alpha\in \N_{\Aut_{\m{D}}(P_0)}(\Aut_{S_0}(P_0))$ extends to an element  $\hat{\alpha}\in \N_{\Aut_{\m{D}}(P_1)}(P_0)$. Then by (\ref{star}), $\hat{\alpha}=c_t\hat{\alpha}_0$ for some $t\in \N_T(P_0)$ and some $\hat{\alpha}_0\in \N_{\Aut_{\F_0}(P_1)}(P_0)$. Hence, $\alpha=\hat{\alpha}|_{P_0}\in \Aut_T(P_0)\Aut_{\F_0}(P_0)$. This yields the assertion.
\end{proof}

\begin{lemma}\label{Lemma4}
 Let $P_0\in\F_0$. Then $O^p(\Aut_\m{D}(P_0))=O^p(\Aut_{\F_0}(P_0))$.
\end{lemma}

\begin{proof}
By \ref{wellplacedExists}, we may choose $U_0\in P_0^\m{D}$ such that $U_0$ is well-placed. By \ref{Lemma3}, $O^p(\Aut_\m{D}(U_0))=O^p(\Aut_{\F_0}(U_0))$. Then for $\beta\in \Iso_\m{D}(U_0,P_0)$, we have $O^p(\Aut_\m{D}(P_0))=O^p(\Aut_\m{D}(U_0))\beta^*=O^p(\Aut_{\F_0}(U_0))\beta^*=O^p(\Aut_{\F_0}(P_0))$ as $\F_0\unlhd \F$. 
\end{proof}

\begin{lemma}\label{Ac}
 Let $P\leq T$. Then $O^p(\Aut_\m{D}(P))\leq \Ac(P)$. In particular, if $P_0\in\F_0^c$, then $O^p(\Aut_\m{D}(P))=\Ac(P)$.
\end{lemma}

\begin{proof}
Let $\vphi\in \Aut_\m{D}(P)$ be a $p^\prime$-element. Observe that $\Aut_{\m{D}/S_0}(PS_0/S_0)=\Aut_{T/S_0}(PS_0/S_0)$ and so $\vphi$ induces a $p$-automorphism of $PS_0/S_0$. Hence, $[P,\vphi]\leq P_0$. Moreover, 
$$\vphi|_{P_0}\in O^p(\Aut_\m{D}(P_0))\leq \Aut_{\F_0}(P_0)$$ 
by \ref{Lemma4}. This proves $\vphi\in\Ac(P)$ which yields the assertion. 
\end{proof}

\section{$\m{D}=\F_0T$ is saturated}\label{S5}

To show that $\m{D}=\F_0T$ is saturated, we assume from now on that $(\F,\F_0,T)$ is a counterexample such that first $\F$ is minimal with respect to inclusion, then $\F_0$ is minimal with respect to inclusion, and then $|T|$ is maximal.

\begin{lemma}
 $T$ is strongly closed in $\F$.
\end{lemma}

\begin{proof}
If $T=S$ then we are done. Thus we may assume that $T<S$ and thus $T<T_1:=\N_S(T)$. Then the maximality of $|T|$ implies that $\F_0T_1$ is saturated. Observe that $T$ is strongly closed in $\F_0T_1$. Therefore, we may assume $\F\neq \F_0T_1$. Then by the minimality of $\F$, $\F_0T=(\F_0T)_{\F_0T_1}$ is saturated, contradicting $(\F,\F_0,T)$ being a counterexample.
\end{proof}

\begin{notation}\label{MainNot}
Set $\m{D}:=\F_0T$. For $U_0\in \F_0^c\cap \F^f$ set
$$\m{D}(U_0):=\N_{N_\F(U_0\C_S(U_0))}(U_0)\mbox{ and }\m{E}(U_0):=\N_{\F_0}(U_0).$$
It follows from \cite[Thm.~2]{AschbacherNormal} that $\m{D}(U_0)$ and $\m{E}(U_0)$ are constrained saturated subsystems of $\F$ with $\m{E}(U_0)\unlhd \m{D}(U_0)$. Hence, by \cite[Thm.~1]{AschbacherNormal} we may choose models $G(U_0)$ and $N(U_0)$ of $\m{D}(U_0)$ respectively $\m{E}(U_0)$ such that $N(U_0)$ is contained in $G(U_0)$ as a normal subgroup.

\smallskip

For $P_0\in \F_0^{fc}$ and $\alpha\in \Hom_\F(\N_S(P_0),S)$ with $P_0\alpha\in \F^f$ set $H(P_0,\alpha):=N(P_0\alpha)(\N_T(P_0)\alpha)\leq G(P_0\alpha)$ and 
$$\m{N}(P_0,\alpha):=\alpha(\F_{\N_T(P_0)\alpha}(H(P_0,\alpha)))\alpha^{-1}.$$
For every $P_0\in \F_0^{fc}$ set
$$\mathfrak{N}(P_0):=\{\m{N}(P_0,\alpha):\alpha\in \Hom_\F(\N_S(P_0),S)\mbox{ such that }P_0\alpha\in \F^f\}.$$
\end{notation}

Note that for $P_0\in\F_0^{fc}$ and $\alpha\in\Hom_\F(N_S(P_0),S)$, $P_0\alpha\in\F_0^c$ by \ref{PrelI}(a), so $G(P_0\alpha)$ and $N(P_0\alpha)$ exist, and $H(P_0,\alpha)$ is well-defined. In fact, for the definition of $H(P_0,\alpha)$ and $\m{N}(P_0,\alpha)$, it would not by necessary to assume $P_0\in\F_0^f$, but this is only to ensure that $\m{N}(P_0,\alpha)$ is saturated, as we prove in detail in the next lemma.

\smallskip

We will use from now on without reference that, by \cite[Lemma~II.3.1]{Aschbacher/Kessar/Oliver:2011a}, for any $P_0\in\F_0^{fc}$, there exists $\alpha\in \Hom_\F(\N_S(P_0),S)$ such that $P_0\alpha\in\F^f$ and in particular, 
$$\mathfrak{N}(P_0)\neq \emptyset.$$ 
(In fact, by \cite[8.3.3]{AschbacherGeneralized}, $|\mathfrak{N}(P_0)|=1$. However, this property will not be needed in our proof.)

\begin{lemma}\label{ElPropN(P_0)}
Let $P_0\in \F_0^{fc}$ and $\m{N}\in\mathfrak{N}(P_0)$. Then $\m{N}$ is a saturated subsystem of $\m{D}$ on $\N_T(P_0)$. Moreover, $\N_{\F_0}(P_0)\unlhd \m{N}$,  $P_0\unlhd \m{N}$ and $O^p(\Aut_{\m{N}}(R))\leq \Ac(R)$ for every $R\leq \N_T(P_0)$.
\end{lemma}

\begin{proof}
Let $\alpha\in \Hom_\F(\N_S(P_0),S)$ such that $P_0\alpha\in\F^f$ and $\m{N}=\m{N}(P_0,\alpha)$. By \ref{PrelI}, $P_0\alpha\in \F_0^{fc}$ and $\N_{S_0}(P_0\alpha)=\N_{S_0}(P_0)\alpha\leq \N_T(P_0)\alpha$. As $N(P_0,\alpha)$ is a model for $\m{E}(P_0\alpha)$, $\N_{S_0}(P_0\alpha)\in \Syl_p(N(P_0,\alpha))$. Hence, $\N_T(P_0)\alpha\in \Syl_p(H(P_0,\alpha))$, so $\m{N}_1:=\F_{\N_T(P_0)\alpha}(H(P_0,\alpha))$ is a saturated fusion system on $\N_T(P_0)\alpha$. In particular, $\m{N}$ is saturated, as $\alpha^{-1}$ induces an isomorphism from $\m{N}_1$ to $\m{N}$. Moreover, $P_0\alpha\unlhd H(P_0,\alpha)$, so $P_0\alpha\unlhd \m{N}_1$ and thus $P_0\unlhd \m{N}$. Note also $\alpha\N_{\F_0}(P_0\alpha)\alpha^{-1}=\N_{\F_0}(P_0)$, whence $\N_{\F_0}(P_0)\unlhd \m{N}$. 
Let  $R\leq \N_T(P_0)$. By \ref{Prop1Str} and \ref{Ac} applied with $(\m{D}(P_0\alpha),\m{E}(P_0\alpha),\m{N}_1)$ in place of $(\F,\F_0,\m{D})$, we get
$$O^p(\Aut_{\m{N}_1}(R\alpha))\leq \Ac_{\m{D}(P_0\alpha),\m{E}(P_0\alpha)}(R\alpha)\leq \Ac(R\alpha).$$
Hence, $O^p(\Aut_\m{N}(R))\leq \Ac(R\alpha)(\alpha^{-1})^*=\Ac(R)$. So it only  remains to show that $\m{N}$ is a subsystem of $\m{D}$. Let $Q\in \m{N}^{frc}$. As $P_0\unlhd \m{N}$, we have, $P_0\leq Q_0$ and so $Q_0\in \F_0^c$ as $P_0\in\F_0^c$. By what we have just shown, $O^p(\Aut_\m{N}(Q))\leq \Ac(Q)\leq \Aut_{\m{D}}(Q)$.  By Alperin's Fusion Theorem \cite[Thm.~A.10]{BrotoLeviOliver2003}, $\m{N}=\<O^p(\Aut_\m{N}(Q)):Q\in\m{N}^{frc}\>_{N_T(P_0)}$, so $\m{N}\subseteq \m{D}$. This shows (b).
\end{proof}

\begin{lemma}\label{Lemma5}
 Let $P_0\in \F_0^{fc}$ be fully $\m{D}$-automized. Then every element $\vphi\in \Aut_\m{D}(P_0)$ extends to $\hat{\vphi}\in \Hom_\m{D}(N_\vphi\cap T,T)$.
\end{lemma}

\begin{proof}
Let $\m{N}\in\mathfrak{N}(P_0)$. By \ref{ElPropN(P_0)}, $\m{N}$ is a saturated subsystem of $\m{D}$ on $\N_S(P_0)$ with $P_0\unlhd \m{N}$ and $\Aut_{\F_0}(P_0)\leq \Aut_\m{N}(P_0)$. Using \ref{Lemma4} and the fact that $P_0$ is fully $\m{D}$-automized, we obtain $\Aut_\m{D}(P_0)=\Aut_{\F_0}(P_0)\Aut_T(P_0)=\Aut_\m{N}(P_0)$. By \ref{NphicapR}, applied with $(T,\m{D})$ in place of $(R,\m{E})$, $\Aut_{N_\vphi\cap T}(P_0)\vphi^*\leq \Aut_T(P_0)=\Aut_{\N_T(P_0)}(P_0)$ and thus $N_\vphi\cap T\leq N_\vphi^\m{N}$, for all $\vphi\in \Aut_\m{D}(P_0)$. Now the assertion follows from the fact that $\m{N}$ is saturated.
\end{proof}

\begin{remark}\label{RemFullAut}
 Let $P\leq T$ be fully $\m{D}$-automized and $Q\in P^\m{D}$. Then there exists $\alpha\in\Iso_\m{D}(Q,P)$ such that $N_T(Q)=N_\alpha\cap T=N_\alpha^\m{D}$.
\end{remark}

\begin{proof}
 For $\beta\in\Iso_\m{D}(Q,P)$, $\Aut_T(Q)\beta^*$ is a $p$-subgroup of $\Aut_\m{D}(P)$. So by Sylow's Theorem, as $\Aut_T(P)\in\Syl_p(\Aut_\m{D}(P))$, there exists $\gamma\in\Aut_\m{D}(P)$ such that $\Aut_T(Q)\beta^*\gamma^*\leq \Aut_T(P)$. Then the assertion holds for $\alpha=\beta\gamma$.
\end{proof}

For the next lemma recall that a subgroup $U\leq T$ is called \textit{$\m{D}$-receptive} if for any $P\leq T$ and $\alpha\in\Iso_\m{D}(P,U)$, $\alpha$ extends to a member of $\Hom_\m{D}(N_\alpha^\m{D},T)$.

\begin{lemma}\label{P0Extend}
Let $U_0\in\F_0^c$ such that $U_0$ is well-placed. Then $U_0$ is $\m{D}$-receptive.
\end{lemma}

\begin{proof}
Assume the assertion is wrong and let $U_0$ be a counterexample such that $|U_0|$ is maximal. We show first:
\begin{eqnarray}\label{P0Extend2}
&&\mbox{Let $P_0\in U_0^\m{D}$ and $\alpha\in \Iso_\m{D}(P_0,U_0)$ such that $ P_0<N_\alpha\cap S_0$.}\\ \nonumber 
&&\mbox{Then $\alpha$ extends to a member of $\Hom_\m{D}(N_\alpha\cap T,T)$.}
\end{eqnarray}
We prove (\ref{P0Extend2}) by contradiction. Let $(P_0,\alpha)$ be a counterexample to (\ref{P0Extend2}) such that $|N_\alpha\cap S_0|$ is maximal. Set $Q_0:=N_\alpha\cap~S_0$. As $U_0$ is well-placed, $U_0$ is fully $\F_0$-normalized, so by \ref{Lemma0}, $\alpha$ extends to $\hat{\alpha}\in \Hom_{\m{D}}(Q_0,S_0)$. Set $R_0:=Q_0\hat{\alpha}$. By \ref{wellplacedExists}, there exists $\tilde{R}_0\in R_0^{\m{D}}$ such that $\tilde{R}_0$ is well-placed. Let $\beta\in \Hom_{\m{D}}(R_0,\tilde{R}_0)$. As $\Aut_T(\tilde{R}_0)\in \Syl_p(\Aut_\m{D}(\tilde{R}_0))$, there exists $\vphi\in \Aut_\m{D}(\tilde{R}_0)$ such that 
\begin{eqnarray*}
\Aut_{N_T(U_0\beta\vphi)}(\tilde{R}_0)&=&\Aut_T(\tilde{R}_0)\cap (\N_{\Aut_{\m{D}}(\tilde{R}_0)}(U_0\beta)\vphi^*)\\
&\in& \Syl_p(\N_{\Aut_{\m{D}}(\tilde{R}_0)}(U_0\beta)\vphi^*)=\Syl_p(\N_{\Aut_\m{D}(\tilde{R}_0)}(U_0\beta\vphi)).
\end{eqnarray*}
So replacing $\beta$ by $\beta\vphi$, we may assume that $\Aut_{\N_T(U_0\beta)}(\tilde{R}_0)\in \Syl_p(\N_{\Aut_\m{D}(\tilde{R}_0)}(U_0\beta)).$ 
Then as $\Aut_{\N_T(U_0)}(R_0)\beta^*\leq \N_{\Aut_\m{D}(\tilde{R}_0)}(U_0\beta)$, there exists $\psi\in \N_{\Aut_\m{D}(\tilde{R}_0)}(U_0\beta)$ such that 
$$\Aut_{\N_T(U_0)}(R_0)\beta^*\psi^*\leq \Aut_{\N_T(U_0\beta)}(\tilde{R}_0).$$ 
Therefore, replacing $\beta$ by $\beta\psi$ we may assume 
$$\N_T(U_0)\cap \N_T(R_0)\leq N_\beta.$$
If $R_0=\N_{S_0}(U_0)$ then, as $U_0$ is well-placed, $R_0$ is also well-placed and, by \ref{WellPlacedDef}(iii), 
$$\Aut_{\N_T(U_0)}(R_0)\in \Syl_p(\N_{\Aut_\m{D}(R_0)}(U_0)).$$ 
Hence, in this case we can and will choose $\tilde{R}_0=R_0$ and $\beta=\id_{R_0}$.

\smallskip

As $U_0$ is well-placed, $U_0$ is fully $\m{D}$-automized. Thus, by \ref{NphicapR}, $N_\alpha\cap T=N_\alpha^\m{D}$. 
Now by \ref{NewLemma} applied with $(\m{D},\alpha\beta|_{U_0},\hat{\alpha}\beta,U_0\beta)$ in place of $(\m{E},\alpha,\hat{\alpha},V_0)$, we have 
$$\Aut_{N_\alpha\cap T}(Q_0)\hat{\alpha}^*\beta^*\leq \Aut_{\N_T(U_0\beta)}(\tilde{R}_0)\C_{\Aut_{\m{D}}(\tilde{R}_0)}(U_0\beta).$$ 
Moreover, as $\Aut_{\N_T(U_0\beta)}(\tilde{R}_0)\in \Syl_p(\N_{\Aut_\m{D}(\tilde{R}_0)}(U_0\beta))$, it follows 
$$\Aut_{\N_T(U_0\beta)}(\tilde{R}_0)\in \Syl_p(\Aut_{\N_T(U_0\beta)}(\tilde{R}_0)\C_{\Aut_{\m{D}}(\tilde{R}_0)}(U_0\beta)).$$
Hence, there exists $\delta\in \C_{\Aut_\m{D}(\tilde{R}_0)}(U_0\beta)$ such that $\Aut_{N_\alpha\cap T}(Q_0)\hat{\alpha}^*\beta^*\delta^*\leq \Aut_{\N_T(U_0\beta)}(\tilde{R}_0)$. Then $N_\alpha\cap T\leq N_{\hat{\alpha}\beta\delta}$. Since $(P_0,\alpha)$ is a counterexample to (\ref{P0Extend2}), $P_0<Q_0$ and thus $|\tilde{R}_0|>|U_0|$. Therefore, as $U_0$ is a counterexample to \ref{P0Extend} with $|U_0|$ maximal, $\hat{\alpha}\beta\delta$ extends to $\gamma\in \Hom_\m{D}(N_\alpha\cap T,T)$. Then $\gamma|_{P_0}=(\hat{\alpha}\beta)|_{P_0}$ since $\delta|_{U_0\beta}=\id$. So $\gamma$ extends $\alpha\beta$. Note also that $|R_0|>|U_0|$, so the maximality of $|U_0|$ gives also that $\beta$ extends to $\ov{\beta}\in \Hom_\m{D}(\N_T(U_0)\cap \N_T(R_0),T)$.

\smallskip

If $R_0=\N_{S_0}(U_0)$ then by our assumption, $\beta=\id$ and $\gamma$ extends $\alpha=\alpha\beta$. Hence, we may assume from now on that $R_0<\N_{S_0}(U_0)$. Then 
$$R_0<\N_{S_0}(U_0)\cap \N_{S_0}(R_0).$$
As $U_0$ is well-placed, $U_0$ is fully $\m{D}$-automized, so \ref{RemFullAut} implies that there exists $\rho_0\in \Hom_\m{D}(U_0\beta,U_0)$ such that $\N_T(U_0\beta)=N_{\rho_0}\cap T$. Then 
$$\tilde{R}_0<(\N_{S_0}(U_0)\cap \N_{S_0}(R_0))\ov{\beta}\leq \N_{S_0}(U_0\beta)=N_{\rho_0}\cap S_0.$$
Hence, $|N_{\rho_0}\cap S_0|>|\tilde{R}_0|=|Q_0|=|N_\alpha\cap S_0|$ and by the maximality of $|N_\alpha\cap S_0|$, $\rho_0$ extends to $\rho\in \Hom_\m{D}(\N_T(U_0\beta),T)$. Set
$$X:=\{t\in \N_T(U_0):c_t|_{U_0}\in \Aut_{\N_\alpha\cap T}(P_0)\alpha^*\}.$$
By \ref{NewLemma}, $X\leq \N_T(U_0)\cap \N_T(R_0)$. Note also $\C_T(U_0)\leq X$. In particular, $\C_T(U_0)\leq \N_T(U_0)\cap \N_T(R_0)$ and hence
$$|\C_T(U_0\beta)|=|\C_T(U_0\beta)\rho|\leq |\C_T(U_0)|=|\C_T(U_0)\ov{\beta}|\leq |\C_T(U_0\beta)|.$$
So equality holds above and thus $\C_T(U_0)\ov{\beta}=\C_T(U_0\beta)$. Now
\begin{eqnarray*}
\Aut_{(N_\alpha\cap T)\gamma}(U_0\beta)&=&\Aut_{N_\alpha\cap T}(P_0)(\gamma|_{P_0})^*\\
&=&\Aut_{N_\alpha\cap T}(P_0)\alpha^*(\beta|_{U_0})^*=\Aut_X(U_0)(\beta|_{U_0})^*=\Aut_{X\ov{\beta}}(U_0\beta),
\end{eqnarray*}
where the first and last equality uses \ref{littleLemma}. This implies 
$$(N_\alpha\cap T)\gamma\leq (X\ov{\beta})\C_T(U_0\beta)=(X\ov{\beta})(\C_T(U_0)\ov{\beta})\leq (\N_T(U_0)\cap \N_T(R_0))\ov{\beta}.$$
Hence, $\gamma\ov{\beta}^{-1}\in \Hom_{\m{D}}(N_\alpha\cap T,T)$ is well-defined and extends $\alpha$, so (\ref{P0Extend2}) holds.

\smallskip

We now derive the final contradiction. Since $U_0$ is a counterexample to the assertion, there exists $P_0\in U_0^\m{D}$ and $\alpha\in \Hom_\m{D}(P_0,U_0)$ such that $\alpha$ does not extend to a member of $\Hom_{\m{D}}(N_\alpha^\m{D},T)$. As $U_0$ is well-placed, $U_0$ is fully $\m{D}$-automized, so by \ref{NphicapR}, $N_\alpha^\m{D}=N_\alpha\cap T$. Observe that $S_0\in\F_0^{fc}$ and $S_0$ is fully $\m{D}$-automized since $\Aut_\m{D}(S_0)=\Aut_{\F_0}(S_0)\Aut_T(S_0)$. So by \ref{Lemma5}, $P_0\neq S_0$ and thus $P_0<\N_{S_0}(P_0)$. As $\Aut_T(U_0)\in \Syl_p(\Aut_\m{D}(U_0))$ there exists $\chi\in \Aut_\m{D}(U_0)$ such that $\Aut_T(P_0)\alpha^*\chi^*\leq \Aut_T(U_0)$. Then $\N_T(P_0)=\N_{\alpha\chi}\cap T$, so $P_0<\N_{S_0}(P_0)=N_{\alpha\chi}\cap S_0$. Thus, by (\ref{P0Extend2}), $\alpha\chi$ extends to an element $\gamma\in \Hom_\m{D}(\N_T(P_0),T)$. Note that 
$$\Aut_{(N_\alpha\cap T)\gamma}(U_0)(\chi^{-1})^*=\Aut_{N_\alpha\cap T}(P_0)(\gamma|_{P_0})^*(\chi^{-1})^*=\Aut_{N_\alpha\cap T}(P_0)\alpha^*\leq \Aut_T(P_0).$$
 Hence, $(N_{\alpha}\cap T)\gamma\leq N_{\chi^{-1}}$. As $U_0\in\F_0^{fc}$ is fully $\m{D}$-automized and $\chi^{-1}\in \Aut_{\m{D}}(U_0)$, it follows from \ref{Lemma5} that $\chi^{-1}$ extends to $\psi\in \Hom_{\m{D}}(N_{\chi^{-1}}\cap T,T)$. Then $\gamma\psi\in \Hom_\m{D}(N_\alpha \cap T,T)$ extends $\alpha$, a contradiction which completes the proof.
\end{proof}

\begin{lemma}\label{wellplacedNormConj}
Let $P_0\in\F_0^c$. Then there exists $\alpha\in \Hom_\m{D}(\N_T(P_0),T)$ such that $P_0\alpha$ is well-placed. 
\end{lemma}

\begin{proof}
By \ref{wellplacedExists}, there exists $U_0\in P_0^\m{D}$ such that $U_0$ is well-placed, and by \ref{RemFullAut}, there exists $\alpha\in\Iso_\m{D}(P_0,U_0)$ with $N_T(P_0)=N_\alpha^\m{D}$. Now the assertion follows from \ref{P0Extend}.                                                                                                                                       
\end{proof}

\begin{lemma}\label{S0leqU}
Let $S_0\leq P\leq T$. Then $\Aut_T(P)\cap \Ac(P)\in \Syl_p(\Ac(P))$.
\end{lemma}

\begin{proof}
Using Notation \ref{MainNot}, let $G:=G(S_0)$ and $N:=N(S_0)$ be models for $\m{D}(S_0)=\N_\F(S_0\C_S(S_0))$ respectively $\m{E}(S_0)=\N_{\F_0}(S_0)$ such that $N\unlhd G$. By \cite[Lemma~II.3.1]{Aschbacher/Kessar/Oliver:2011a}, there exists $\beta\in \Hom_\F(\N_S(P),S)$ such that $Q:=P\beta\in \F^f$. Set 
$$T_1:=\N_T(P)\beta,\;H:=\N_N(Q)T_1\mbox{ and } \m{G}:=\F_{T_1}(H).$$
As $S_0\leq N\cap Q\leq \N_N(Q)$ and $S_0\in \Syl_p(N)$, we have $S_0\in \Syl_p(\N_N(Q))$. Moreover, $S_0=S_0\beta\leq \N_T(P)\beta=T_1$, so $T_1\in \Syl_p(H)$ and $\m{G}$ is saturated. Assume first
\begin{equation}\label{AcleqD}
\Ac(Q)\leq \Aut_{\m{D}(S_0)}(Q).
\end{equation}
Then $\Ac(Q)=\Ac_{\m{D}(S_0),\m{E}(S_0)}(Q)=O^p(\Aut_N(Q))=O^p(\Aut_H(Q))=O^p(\Aut_{\m{G}}(Q))$, where the first equality uses (\ref{AcleqD}), the second uses \ref{Prop1Str}, and the third uses $O^p(H)=O^p(\N_N(Q))$. As $Q\unlhd H$, we have $Q\unlhd \m{G}$. In particular $Q\in\m{G}^f$ and, as $\m{G}$ is saturated, $\Aut_{T_1}(Q)\in \Aut_{\m{G}}(Q)$. Hence,
$$\Ac(Q)\cap \Aut_{T_1}(Q)\in \Syl_p(\Ac(Q)).$$
Observe that $\Ac(P)=\Ac(Q)(\beta^{-1})^*$ and, by \ref{littleLemma}, $\Aut_{T_1}(Q)(\beta^{-1})^*=\Aut_{T_1\beta^{-1}}(P)=\Aut_T(P)$. So the assertion follows and it remains only to prove (\ref{AcleqD}).
For the proof let $\vphi\in \Ac(Q)$ be a $p^\prime$-element. Then $\vphi_0:=\vphi|_{S_0}\in \Aut_{\F_0}(S_0)$ and, by \ref{littleLemma},  $Q\C_S(S_0)\leq N_{\vphi_0}$. Observe that $\Aut_{\F_0S}(S_0)=\Aut_{\F_0}(S_0)\Aut_S(S_0)$ and thus $S_0$ is fully $\F_0S$-automized. Hence, it follows from \ref{Lemma5} that $\vphi_0$ extends to $\psi\in \Hom_{\F_0S}(Q\C_S(S_0),S)$. By construction of $\F_0S$, $\psi=\chi c_s$ for some $s\in S$ and $\chi\in \Hom_{\F_0S}(Q\C_S(S_0),S)$ with $[Q\C_S(S_0),\chi]\leq S_0$ and $\chi|_{S_0}\in \Aut_{\F_0}(S_0)$. Then $c_s|_{S_0}=(\chi|_{S_0})^{-1}\vphi_0\in \Aut_S(S_0)\cap \Aut_{\F_0}(S_0)=\Inn(S_0)$, so $s=s_0c$ for some $s_0\in S_0$ and $c\in \C_S(S_0)$. Observe now that $\psi_1:=\chi c_{s_0}$ also extends $\vphi_0$, so replacing $(\psi,s)$ by $(\psi_1,s_0)$ we may assume $s\in S_0$. Then $[Q\C_S(S_0),\psi]\leq S_0$, hence we have $\vphi\psi^{-1}|_{S_0}=\id$ and $[Q,\vphi\psi^{-1}]\leq S_0$. By \cite[8.2.2(b)]{KurzweilStellmacher2004}, this yields $\vphi\psi^{-1}\in O_p(\Aut_\F(Q))\leq \Aut_S(Q)\leq \Aut_{\m{D}(S_0)}(Q)$, where we use $\Aut_S(Q)\in \Syl_p(\Aut_\F(Q))$ as $Q\in\F^f$. Since $\psi$ is a morphism in $\m{D}(S_0)$, it follows $\vphi\in \Aut_{\m{D}(S_0)}(Q)$ showing (\ref{AcleqD}). This completes the proof. 
\end{proof}

For the proof of the next lemma recall the definition of $K$-normalizers and fully $K$-normalized subgroups from \cite[Section~I.5]{Aschbacher/Kessar/Oliver:2011a}

\begin{lemma}\label{AcleqAutN}
Let $U\in \m{D}$ such that $U_0\in \F_0^{fc}$. Let $\m{N}\in\mathfrak{N}(U_0)$ and $R\in U^{\m{N}}\cap \m{N}^f$. Then $\Ac(R)\leq \Aut_{\m{N}}(R)$. 
\end{lemma}
 
\begin{proof}  
Let $\alpha\in \Hom_\F(\N_S(U_0),S)$ with $U_0\alpha\in\F^f$ and $\m{N}=\m{N}(U_0,\alpha)$. Then $\Ac(R)\leq \Aut_{\m{N}}(R)$ is equivalent to $\Ac(P)\leq \Aut_{\m{N}_1}(P)$ for $P:=R\alpha$, $T_1:=\N_T(U_0)\alpha$ and $\m{N}_1:=\F_{T_1}(H(U_0,\alpha))$. Since $U_0\unlhd \m{N}$, $R_0=U_0$ and so $P_0=R_0\alpha=U_0\alpha$. Assume by contradiction that there exists a $p^\prime$-element $\vphi\in \Ac(P)$ with $\vphi\not\in \Aut_{\m{N}_1}(P)$. Set 
$$\vphi_0:=\vphi|_{P_0},\;K_0:=\Inn(P_0)\<\vphi_0\>,\;K:=\Aut_P(P_0)\<\vphi_0\>,\;\m{G}_0:=\N_{\F_0}^{K_0}(P_0),\mbox{ and }\m{G}:=\N_\F^K(P_0).$$
Note that $[\Aut_P(P_0),\vphi_0]\leq \Inn(P_0)$ by \ref{littleLemma} as $[P,\vphi]\leq P_0$. In particular, $K_0\unlhd K$ and $\Aut_P(P_0)\unlhd K$.
As $\vphi_0$ is a $p^\prime$-element, we get $\Aut_P(P_0)\in \Syl_p(K)$ and $\Inn(P_0)\in \Syl_p(K_0)$. This yields $\N_S^K(P_0)=P\C_S(P_0)$ and  $\N_{S_0}^{K_0}(P_0)=P_0\C_{S_0}(P_0)=P_0$. Moreover, 
$$\Aut_S^K(P_0)=\Aut_P(P_0)\in \Syl_p(K)=\Syl_p(\Aut_\F^K(P_0))$$
and
$$\Aut_{S_0}^{K_0}(P_0)=\Inn(P_0)\in \Syl_p(K_0)=\Syl_p(\Aut_{\F_0}^{K_0}(P_0)).$$
Since $P_0\in\F^f$, $P_0$ is fully $\F$-centralized and, as $U_0\in\F_0^f$, it follows from \ref{PrelI}(b) that $P_0\in\F_0^f$. 
Hence, by \cite[Prop.~I.5.2]{Aschbacher/Kessar/Oliver:2011a}, $P_0$ is fully $K$-normalized in $\F$ and fully $K_0$-normalized in $\F_0$. Now \cite[Thm.~I.5.5]{Aschbacher/Kessar/Oliver:2011a} implies: 
\begin{equation}\label{1}
\m{G}\mbox{ and }\m{G}_0\mbox{ are saturated subsystems of }\F\mbox{ on }P\C_S(P_0)\mbox{ respectively on }P_0. 
\end{equation}
We show next:
\begin{equation}\label{2}
\m{G}_0\unlhd \m{G}.
\end{equation}
Observe that $\m{G}_0$ is $\m{G}$-invariant as $K_0\unlhd K$. By (\ref{1}), $\m{G}$ and $\m{G}_0$ are saturated. Furthermore, clearly every element $c_x\in \Inn(P_0)$ with $x\in P_0$ extends to an element $c_x\in \Aut_\m{G}(P_0\C_S(P_0))$ and $[\C_S(P_0),c_x]\leq [\C_S(P_0),P_0]=1$. Hence, it is sufficient to prove that $\vphi_0$ extends to $\ov{\vphi}\in \Aut_\m{G}(P_0\C_S(P_0))$ with $[\C_S(P_0),\ov{\vphi}]\leq Z(P_0)$. To show that, set $H_0:=N(P_0)\N_S(P_0)\leq G(P_0)$ and note that $\F_{\N_S(P_0)}(H_0)$ is saturated, as $\N_S(P_0)\in \Syl_p(H_0)$. Hence, $\vphi_0$ extends to a $p^\prime$-element $\ov{\vphi}\in \Aut_{H_0}(P_0\C_S(P_0))$. Then $\ov{\vphi}\in O^p(\Aut_{H_0}(P_0\C_S(P_0)))\leq \Ac(P_0\C_S(P_0))$ by  \ref{Prop1Str}. Hence, $[\C_S(P_0),\ov{\vphi}]\leq P_0\cap \C_S(P_0)=Z(P_0)$. This proves (\ref{2}). We show next the following property:
\begin{equation}\label{3}
O_p(\Ac_{\m{G},\m{G}_0}(P))\not\leq \Aut_{T_1}(P).
\end{equation}
For the proof note first that $\vphi_0\in \Aut_{\F_0}(P_0)\leq \Aut_{\m{N}_1}(P_0)$ and $\m{N}_1$ is saturated. Hence, as $P\C_{T_1}(P_0)\leq N_{\vphi_0}^{\m{N}_1}$ by \ref{littleLemma}, $\vphi_0$ extends to $\psi\in \Hom_{\m{N}_1}(P\C_{T_1}(P_0),T_1)$. By \ref{CommInC(P0)}, we have $[\C_P(\vphi),\psi]\leq \C_{T_1}(P_0)$. As the action of $\vphi$ on $P$ is coprime and $[P,\vphi]\leq P_0$, \cite[8.2.7(a)]{KurzweilStellmacher2004} yields $P=P_0\C_P(\vphi)$. Hence, $\psi\in \Aut_{\m{N}_1}(P\C_{T_1}(P_0))$. As $\vphi_0$ is a $p^\prime$-element, we can then choose $\psi$ to be a $p^\prime$-element. So $\psi\in O^p(\Aut_{\m{N}_1}(P\C_{T_1}(P_0)))\leq \Ac(P\C_{T_1}(P_0))$ by \ref{Prop1Str}. In particular, $[P,\psi]\leq S_0$ and thus $[\C_P(\vphi),\psi]\leq \C_{S_0}(P_0)=P_0$. Hence as $P=P_0C_P(\vphi)$, $P\psi=P$ and $\psi|_P\in \Ac(P)\cap \Aut_{\m{N}_1}(P)$. In particular, $\vphi,\psi|_P\in \Ac_{\m{G},\m{G}_0}(P)$ and $[P,\vphi(\psi|_P)^{-1}]\leq P_0$. As $\vphi|_{P_0}(\psi|_{P_0})^{-1}=\id_{P_0}$, it follows from \cite[8.2.2(b)]{KurzweilStellmacher2004} that $\vphi(\psi|_P)^{-1}\in O_p(\Ac_{\m{G},\m{G}_0}(P))$. By assumption, $\vphi\not\in \Aut_{\m{N}_1}(P)$, so $\vphi(\psi|_P)^{-1}\not\in \Aut_{\m{N}_1}(P)$ and in particular,
$\vphi(\psi|_P)^{-1}\not\in \Aut_{T_1}(P)$. This proves (\ref{3}).

\smallskip

We now derive the final contradiction. If $(\m{G}_0P)_\m{G}$ is saturated, then 
$$\Inn(P)\cap \Ac_{\m{G},\m{G}_0}(P)\in \Syl_p(\Ac_{\m{G},\m{G}_0}(P))$$
which contradicts (\ref{3}). Hence, because of the minimality of $\F$ and $\F_0$, $\m{G}=\F$ and $\m{G}_0=\F_0$. In particular, $P_0=S_0\leq P$. Hence, by \ref{S0leqU}, $O_p(\Ac(P))\leq \Aut_T(P)$. As $U_0=S_0\unlhd S$, we get also $T_1=T\alpha=T$. Hence we have again a contradiction to (\ref{3}). This completes the proof.
\end{proof}

\begin{notation}
 Set
\begin{eqnarray*}
\m{H}_0&:=&\{P\leq T:P_0\in\F_0^c\},\\
\m{H}&:=&\m{H}_0\cap \m{D}^c,\\
\m{G}_0&:=&\{P\leq T:P\in\m{D}^f\mbox{ and }P_0\in\F_0^{fc}\},\\
\m{G}&:=&\m{G}_0\cap \m{D}^c.
\end{eqnarray*}
Furthermore set $\A(P):=\Aut_T(P)\Ac(P)$ for any $P\leq T$.
\end{notation}

\begin{lemma}\label{G0Good}
 Let $U\in\m{G}_0$ and $\m{N}\in\m{N}(U_0)$. Then $\Aut_\m{N}(U)=\A(U)$, $\Aut_T(U)\in \Syl_p(\A(U))$, and every element $\vphi\in \A(U)$ extends to an element of $\Hom_\m{D}(N_\vphi\cap T,T)$.
\end{lemma}

\begin{proof}
By \ref{ElPropN(P_0)}, $\m{N}$ is a saturated subsystem of $\m{D}$. As $U\in\m{D}^f$ and $\m{N}$ is a subsystem of $\m{D}$ on $\N_T(U_0)\geq \N_T(U)$, it follows $U\in\m{N}^f$ and $\Aut_T(U)\in \Syl_p(\Aut_\m{N}(U))$. By \ref{ElPropN(P_0)} and  \ref{AcleqAutN}, $O^p(\Aut_\m{N}(U))=\Ac(U)$, which implies $\A(U)=\Aut_\m{N}(U)$. Since $U$ is fully automized in $\m{N}$, by \ref{NphicapR}, $N_\vphi\cap T=N_\vphi\cap N_T(U_0)=N_\vphi^{\m{N}}$ for any $\vphi\in\A(U)$. Now the assertion follows from the fact that $\m{N}$ is saturated.
\end{proof}

\begin{lemma}\label{G0GoodCons}
 Let $U\in\m{G}_0$ and $\vphi\in \A(U)$. Then there exists $\chi\in \Ac(U)$ such that $\vphi\chi$ extends to a member of $\Aut_\m{D}(\N_T(U))$.
\end{lemma}

\begin{proof}
 By \ref{G0Good}, $\Aut_T(U)\in \Syl_p(\A(U))$. So as $\Aut_T(U)\vphi^*$ is a $p$-subgroup of $\A(U)$, there exists $\chi\in\Ac(U)$ such that $\Aut_T(U)(\vphi\chi)^*=\Aut_T(U)\vphi^*\chi^*\leq \Aut_T(U)$. Then $\N_T(U)=N_{\vphi\chi}\cap T$, so again by \ref{G0Good}, the assertion follows. 
\end{proof}

\begin{lemma}\label{ConjH1}
 Let $P\in \m{H}_0$. Then $P^\m{D}\cap \m{G}_0\neq\emptyset$. In particular, $P^\m{D}\cap \m{G}\neq \emptyset$ for $P\in\m{H}$.
\end{lemma}

\begin{proof}
Let $R\in P^\m{D}\cap \m{D}^f$. By \ref{wellplacedNormConj}, there exists  $\alpha\in \Hom_\m{D}(\N_T(R_0),T)$ such that $Q_0:=R_0\alpha$ is well-placed. Note that $R\leq \N_T(R)\leq \N_T(R_0)$. So $Q:=R\alpha$ is well-defined, and $Q\in\m{D}^f$ as $R\in\m{D}^f$. Moreover, $Q\in R^{\m{D}}=P^{\m{D}}$ and $Q_0\in\F_0^f$ as $Q_0$ is well-placed. By \ref{PrelI}(a), $Q_0\in\F_0^c$ as $P_0\in\F_0^c$. This proves the assertion.
\end{proof}

\begin{lemma}\label{DGenG}
 We have $\m{D}=\<\Ac(P):P\in \m{G}\>_T$. In particular, $\m{D}=\<\Ac(P):P\in \m{H}\>_T$.
\end{lemma}

\begin{proof}
Set $\m{D}_0:=\<\Ac(P):P\in\m{G}\>_T$ and assume $\m{D}_0\neq \m{D}$. By definition of $\m{D}$, there exists then $P\in\m{H}_0$ such that $\Ac(P)\not\leq \Aut_{\m{D}_0}(P)$. We choose $P$ such that $|P|$ is maximal subject to these properties. We show first:
\begin{equation}\label{Gen1}
P^{\m{D}}=P^{\m{D}_0}. \mbox{ In particular, }\Ac(Q)\not\leq \Aut_{\m{D}_0}(Q)\mbox{ for all }Q\in P^{\m{D}}.
\end{equation}
For the proof of (\ref{Gen1}) let $Q\in P^{\m{D}}$ and $\vphi\in \Iso_\m{D}(P,Q)$. We will show that $Q\in P^{\m{D}_0}$. By definition of $\m{D}$, there exists $P_1,\dots,P_n\in\m{H}_0$, $\vphi_i\in\Ac(P_i)$ and $t\in T$ such that $\vphi=\vphi_1\dots\vphi_n c_t$. As $c_t$ is a morphism in $\m{D}_0$, we may assume that $t=1$. Set now $\psi:=\prod_{i\leq n, |P_i|>|P|}\vphi_i$. Observe  that $\psi\in \Hom_\m{D}(P,Q)$ is a well-defined morphism. Because of the maximality of $|P|$, $\vphi_i$ is a $\m{D}_0$-morphism, for every $i\leq n$ with $|P_i|>|P|$. Hence, $\psi\in \Hom_{\m{D}_0}(P,Q)$ and $Q\in P^{\m{D}_0}$. This proves (\ref{Gen1}).

\smallskip

By \ref{PrelI}(a), $\m{H}_0$ is invariant under taking $\F$-conjugates. Hence, by (\ref{Gen1}), we may replace $P$ by any $\m{D}$-conjugate of $P$. 
By \ref{ConjH1}, there exists $Q\in P^{\m{D}}\cap \m{G}_0$, so replacing $P$ by $Q$ we may assume $P\in\m{G}_0$. If $P\in\m{G}$ then, by definition of $\m{D}_0$, $\Ac(P)\leq \Aut_{\m{D}_0}(P)$ contradicting the choice of $P$. Hence, as $P\in\m{G}_0$, $P\not\in\m{D}^c$, i.e. we can choose $U\in P^\m{D}$ such that $\C_T(U)\not\leq U$. By \ref{wellplacedNormConj}, there exists $\xi\in \Hom_{\m{D}}(\N_T(U_0),T)$ such that $U_0\xi$ is well-placed. Then $\C_T(U\xi)\geq \C_T(U)\xi\not\leq U\xi$. Thus, replacing $U$ by $U\xi$, we may assume that $U_0$ is well-placed and, in particular, $U_0\in\F_0^f$. Then by \ref{PrelI}(a), $U_0\in \F_0^{fc}$. Let $\m{N}\in\mathfrak{N}(U_0)$. Then by \ref{ElPropN(P_0)}, $\m{N}$ is a saturated subsystem of $\m{D}$ on $\N_T(U_0)$ with $U_0\unlhd \m{N}$. In particular, by \cite[Lemma~II.3.1]{Aschbacher/Kessar/Oliver:2011a}, there exists $\gamma\in \Hom_\m{N}(\N_T(U),\N_T(U_0))$ such that $R:=U\gamma\in\m{N}^f$. Then $\C_T(R)\geq \C_T(U)\gamma\not\leq U\gamma=R$ and thus $R<\wt{R}:=R\C_T(R)$. The maximality of $|P|$ yields now $\Ac(\wt{R})\leq \Aut_{\m{D}_0}(\wt{R})$. Let $\alpha\in \Ac(R)$ be a $p^\prime$-element. By \ref{AcleqAutN}, $\alpha\in \Ac(R)\leq \Aut_{\m{N}}(R)$. So as $\m{N}$ is saturated, $\alpha$ extends to $\hat{\alpha}\in \Aut_\m{N}(\wt{R})$. As $\alpha$ is a $p^\prime$-element, we can choose $\hat{\alpha}$ to be a $p^\prime$-element. Then $\hat{\alpha}\in O^p(\Aut_\m{N}(\wt{R}))\leq \Ac(\wt{R})\leq \Aut_{\m{D}_0}(\wt{R})$ by \ref{ElPropN(P_0)}. Hence, $\alpha=\hat{\alpha}|_R\in \Aut_{\m{D}_0}(R)$. This shows $\Ac(R)\leq \Aut_{\m{D}_0}(R)$. As $R\in P^\m{D}$ this is a contradiction to (\ref{Gen1}).
\end{proof}

\begin{lemma}\label{MainLemma}
 Let $P\in\m{H}$. Then $P$ is $\m{D}$-receptive and, if $P\in\m{D}^f$, then $\Aut_T(P)\in \Syl_p(\Aut_\m{D}(P))$ and $\Aut_\m{D}(P)=\A(P)$. 
\end{lemma}

\begin{proof}
For the proof note first that, by \ref{Ac}, for any $P\in\m{H}$, we have $\Aut_\m{D}(P)=\A(P)$ provided $\Aut_T(P)\in \Syl_p(\Aut_\m{D}(P))$. So assuming the assertion is wrong, there exists $P\in\m{H}$ such that $P$ is not $\m{D}$-receptive, or $P\in\m{D}^f$ and $P$ is not fully $\m{D}$-automized. In particular, there exists $X\in \m{H}$ such that one of the following holds:
\begin{itemize}
 \item [(i)] $X$ is not $\m{D}$-receptive.
 \item [(ii)] There exists a fully normalized $\m{D}$-conjugate of $X$ which is not fully automized.
\end{itemize}
We choose such $X$ of maximal order. By \ref{ConjH1}, there exists $U\in\m{G}\cap X^\m{D}$. The maximality of $|X|=|U|$ yields:
\begin{eqnarray}\label{ML2}
&&\mbox{For any }Y\in\m{H}\mbox{ with $|Y|>|U|$, $Y$ is $\m{D}$-receptive and, if $Y\in\m{D}^f$,}\\ \nonumber
&&\mbox{then $Y$ is fully $\m{D}$-automized and }\Aut_\m{D}(Y)=\A(Y).
\end{eqnarray}
Next we show the following property:
\begin{eqnarray}\label{ML1}
 \mbox{$U$ is not fully $\m{D}$-automized or not $\m{D}$-receptive.}
\end{eqnarray}
For the proof of (\ref{ML1}) note first that $U$ is fully $\m{D}$-centralized as $U\in\m{D}^c$. Hence, if (ii) holds, then by \cite[Lemma~2.3(a)]{BCGLO2005} applied with $X^\m{D}$ in place of $\m{H}$, $U$ is not fully $\m{D}$-automized. If (ii) is false, then in particular, $U$ is fully $\m{D}$-automized. Moreover, (i) holds, so as $X\in\m{D}^c$ is fully $\m{D}$-centralized, by \cite[Lemma~I.2.6(c)]{Aschbacher/Kessar/Oliver:2011a}, $U$ is not $\m{D}$-receptive. This proves (\ref{ML1}).

\smallskip

Clearly $T$ is $\m{D}$-receptive. Moreover, by construction of $\m{D}$, $\Aut_\m{D}(T)=\Inn(T)\Ac(T)=\A(T)$. So by \ref{G0Good}, $T$ is fully $\m{D}$-automized. This shows:
\begin{equation}\label{ML3}
 U\neq T.
\end{equation}
We show next: 
\begin{eqnarray}\label{ML4}
&&\mbox{Let $V\in U^\m{D}$, $V<V_1\leq \N_T(V)$ and $\alpha\in\Iso_\m{D}(V,U)$ such that $\alpha$ extends}\\ \nonumber  
&&\mbox{to an element }\hat{\alpha}\in \Hom_\m{D}(V_1,T).\mbox{ Then there exists }\chi\in\Ac(U)\mbox{ such that}\\ \nonumber  
&&\mbox{$\alpha\chi$ extends to an element of $\Hom_\m{D}(\N_T(V),\N_T(U))$.} 
\end{eqnarray}
We prove (\ref{ML4}) by contradiction. Let $(V,V_1,\alpha,\hat{\alpha})$ be a counterexample to (\ref{ML4}) such that first $|V_1|$, then the order of $V_2:=\N_T(V)\cap \N_T(V_1)$, and then the order of $\N_T(U)\cap \N_T(V_1\hat{\alpha})$ is maximal. Set 
$$U_1:=V_1\hat{\alpha}\mbox{ and }U_2:=\N_T(U)\cap \N_T(U_1).$$
As $(V,V_1,\alpha,\hat{\alpha})$ is a counterexample, $V_1<\N_T(V)$ and thus $V_1<V_2$. As $U\in\m{D}^f$, $|\N_T(U)|\geq |\N_T(V)|>|V_1|=|U_1|$. Thus $U_1<\N_T(U)$ and $U_1<U_2$.

\smallskip

Let $R_1\in U_1^\m{D}\cap \m{D}^f$ and $\beta\in \Hom_\m{D}(U_1,R_1)$. As $R_1\in\m{D}^f$ and $|R_1|=|V_1|>|V|=|U|$, it follows from (\ref{ML2}) that $\Aut_T(R_1)\in \Syl_p(\Aut_\m{D}(R_1))$. Hence, there exists $\mu\in \Aut_\m{D}(R_1)$ such that $\N_{\Aut_\m{D}(U_1)}(U)\beta^*\mu^*\cap \Aut_T(R_1)\in \Syl_p(\N_{\Aut_\m{D}(U_1)}(U)\beta^*\mu^*)$. So replacing $\beta$ by $\beta\mu$, we can assume $\N_{\Aut_\m{D}(U_1)}(U)\beta^*\cap \Aut_T(R_1)\in \Syl_p(\N_{\Aut_\m{D}(U_1)}(U)\beta^*)$. Setting
$$R:=U\beta\mbox{ and }R_2:=\N_T(R)\cap \N_T(R_1)$$
this gives
$$\Aut_{R_2}(R_1)=\N_{\Aut_\m{D}(R_1)}(R)\cap \Aut_T(R_1)\in \Syl_p(\N_{\Aut_\m{D}(R_1)}(R)).$$
As $\Aut_{U_2}(U_1)\beta^*$ is a $p$-subgroup of $\N_{\Aut_\m{D}(R_1)}(R)$, there exists $\eta\in \N_{\Aut_\m{D}(R_1)}(R)$ such that $\Aut_{U_2}(U_1)\beta^*\eta^*\leq \Aut_{R_2}(R_1)$. Hence, replacing $\beta$ by $\beta\eta$, we may assume
$$\Aut_{U_2}(U_1)\beta^*\leq \Aut_{R_2}(R_1).$$
Then as $|U|<|U_1|$, it follows from (\ref{ML2}) that $\beta$ extends to $\hat{\beta}\in \Hom_\m{D}(U_2,R_2)$. As $\Aut_{V_2}(V_1)\hat{\alpha}^*\beta^*$ is a $p$-subgroup of $\N_{\Aut_\m{D}(R_1)}(R)$, there exists $\rho\in O^p(\N_{\Aut_\m{D}(R_1)}(R))$ such that $\Aut_{V_2}(V_1)\hat{\alpha}^*\beta^*\rho^*\leq \Aut_{R_2}(R_1)$. Then $\hat{\alpha}\beta\rho\in \Hom_\m{D}(V_1,R_1)$ and $V_2\leq \N_{\hat{\alpha}\beta\rho}^\m{D}$, so again by (\ref{ML2}), $\hat{\alpha}\beta\rho$ extends to an element $\gamma\in \Hom_\m{D}(V_2,R_2)$. Then $\gamma$ extends also $\alpha \beta|_U\rho|_R$. Note that $\rho|_R\in O^p(\Aut_\m{D}(R))$ and thus also $\psi:=\rho|_R(\beta|_U^{-1})^*\in O^p(\Aut_\m{D}(U))=\Ac(U)$ by \ref{Ac}. Observe furthermore $\alpha \beta|_U\rho|_R=\alpha\psi\beta|_U$.

\smallskip

Assume first $V_2\gamma=R_2$. Then setting $\wt{V_2}:=U_2\hat{\beta}\gamma^{-1}$, it follows that $\gamma|_{\wt{V_2}}\hat{\beta}^{-1}\in \Iso_\m{D}(\wt{V_2},U_2)$ extends $\alpha\psi\in \Iso_\m{D}(V,U)$. Recall $|V_1|=|U_1|<|U_2|=|\wt{V_2}|$. Hence, the maximality of $|V_1|$ yields the existence of $\chi_0\in \Ac(U)$ such that $\alpha\psi\chi_0$ extends to an element of $\Hom_\m{D}(\N_T(V),\N_T(U))$. Then (\ref{ML4}) holds with $\chi:=\psi\chi_0\in \Ac(U)$ and so $(V,V_1,\alpha,\hat{\alpha})$ is not a counterexample. This shows $V_2\gamma\neq R_2$ and thus
$$|V_2|<|R_2|.$$
Note that $\wt{\alpha}:=\beta^{-1}|_R\in \Hom_\m{D}(R,U)$ extends to $\beta^{-1}\in \Hom_\m{D}(R_1,U_1)$. Hence, as $|V_2|<|R_2|$, the maximality of $|V_2|$ yields the existence of $\wt{\chi}\in\Ac(U)$ such that $\wt{\alpha}\wt{\chi}$ extends to an element $\vphi\in \Hom_\m{D}(\N_T(R),\N_T(U))$. Then $R\vphi=U$ and $\beta|_U\vphi|_R=\wt{\alpha}^{-1}\vphi|_R=\wt{\chi}\in\Ac(U)$.

\smallskip

Suppose first $|\N_T(U)\cap \N_T(R_1\vphi)|>|U_2|$. Note that $\alpha\wt{\chi}=\alpha\beta|_U\vphi|_R\in \Iso_\m{D}(V,U)$ extends to $\hat{\alpha}\beta\vphi\in \Hom_\m{D}(V_1,T)$ and $V_1\hat{\alpha}\beta\vphi=R_1\vphi$. Now the maximality of the order of $U_2=\N_T(U)\cap \N_T(V_1\hat{\alpha})$ and the assumption $|\N_T(U)\cap \N_T(R_1\vphi)|>|U_2|$ implies that there exists $\chi_0\in\Ac(U)$ such that $\alpha\wt{\chi}\chi_0$ extends to an element of $\Hom_\m{D}(\N_T(V),\N_T(U))$. Hence, the claim holds with $\chi:=\wt{\chi}\chi_0$, so $(V,V_1,\alpha,\hat{\alpha})$ is not a counterexample. This shows
$$|\N_T(U)\cap \N_T(R_1\vphi)|=|U_2|.$$
Therefore
$$|R_2|=|R_2\vphi|\leq |\N_T(U)\cap \N_T(R_1\vphi)|=|U_2|=|U_2\hat{\beta}|\leq |R_2|.$$
Now equality holds above, so $|U_2|=|R_2|$ and $U_2\hat{\beta}=R_2$. Recall that $\alpha\psi\beta|_U$ extends to $\gamma\in \Hom_\m{D}(V_2,R_2)$. Therefore, $\alpha\psi$ extends to $\gamma\hat{\beta}^{-1}\in \Hom_\m{D}(V_2,U_2)$. As $V_1<V_2$, the maximality of $|V_1|$ yields that there exists $\chi_0\in \Ac(U)$ such that $\alpha\psi\chi_0$ extends to an element of $\Hom_\m{D}(\N_T(V),\N_T(U))$. Now it follows with $\chi:=\psi\chi_0\in\Ac(U)$ that $(V,V_1,\alpha,\hat{\alpha})$ is not a counterexample. This final contradiction proves (\ref{ML4}). We show next:
\begin{eqnarray}\label{ML5}
&&\mbox{Let $V\in U^\m{D}$ and $\alpha\in \Hom_\m{D}(V,U)$. Then there exists $\chi\in\Ac(U)$} \\ \nonumber
&&\mbox{such that $\alpha\chi$ extends to an element of $\Hom_\m{D}(\N_T(V),\N_T(U))$.} 
\end{eqnarray}
By \ref{DGenG}, there exist $P_1,\dots,P_m\in\m{G}$ and, for $1\leq i\leq m$,  $\phi_i\in A(P_i)$ such that $\alpha=\phi_1\dots\phi_m$. More precisely,  setting $V_1:=V$, $V_{i+1}:=V_i\phi_i$ and $\varphi_i:=\phi_i|_{V_i}$, we have  $\alpha=\vphi_1\dots\vphi_m$. We will prove the following generalization of (\ref{ML5}):
\begin{itemize}
 \item [(*)] For each $1\leq k\leq m$, there exists $\chi_k\in \Ac(U)$ such that $\vphi_k\dots\vphi_m\chi_k$ extends to an element of $\Hom_\m{D}(\N_T(V_k),\N_T(U))$.
\end{itemize}
To prove (*) consider first the case $m=k$. If $P_m=V_m$ then $U=P_m=V_m$ and $\vphi_m\in A(U)$, so (*) follows from \ref{G0GoodCons}. If $V_m<P_m$ then $V_m<W_m:=\N_{P_m}(V_m)$. Hence, $\vphi_m$ extends to $\phi_m|_{W_k}\in \Hom_\m{D}(W_k,T)$ and (*) follows from (\ref{ML4}). So by induction on $m-k$ we may assume from now on that $k<m$ and for $\mu:=\vphi_{k+1}\dots\vphi_m$, there exists $\chi\in \Ac(U)$ such that $\mu\chi$ extends to an element $\beta\in \Hom_\m{D}(\N_T(V_{k+1}),\N_T(U))$. If $V_k<P_k$ then $V_k<W_k:=\N_{P_k}(V_k)$ and $\vphi_k\mu\chi\in \Hom_\m{D}(V_k,U)$ extends to $\phi_k|_{W_k}\beta\in \Hom_\m{D}(W_k,T)$. Hence, by (\ref{ML4}), there exists $\chi_0\in \Ac(U)$ such that $\vphi_k\mu\chi\chi_0$ extends to an element of $\Hom_\m{D}(\N_T(V_k),\N_T(U))$. Thus, (*) holds in this case for $\chi_k:=\chi\chi_0$. Assume now $V_k=P_k$. Then $V_k=V_{k+1}=P_k\in\m{G}$ and $\vphi_k\in A(V_k)$. Hence, by \ref{G0GoodCons}, there exists $\rho\in \Ac(V_k)$ such that $\vphi_k\rho$ extends to an element $\gamma\in \Aut_\m{D}(\N_T(V_k))$. Then $\gamma\beta\in \Hom_\m{D}(\N_T(V_k),\N_T(U))$  extends $\vphi_k\rho\mu\chi=\vphi_k\mu(\rho\mu^*)\chi$. Observe that $\rho\mu^*\in \Ac(U)$, so (*) holds with $\chi_k:=(\rho\mu^*)\chi\in \Ac(U)$. This completes the proof (*) and thus of (\ref{ML5}). We show next:
\begin{equation}\label{ML6}
 \mbox{$U$ is $\m{D}$-receptive.}
\end{equation}
For the proof of (\ref{ML6}) let $V\in U^{\m{D}}$ and $\alpha\in \Hom_\m{D}(V,U)$. By (\ref{ML5}), there exists $\chi\in \Ac(U)$ such that $\alpha\chi$ extends to $\beta\in \Hom_\m{D}(\N_T(V),\N_T(U))$. As $U\in\m{G}$, it follows from \ref{G0Good} that $\chi^{-1}$ extends to $\eta\in \Hom_\m{D}(N_{\chi^{-1}}\cap T,T)$. By \ref{littleLemma}, $\Aut_{N_\alpha^\m{D}}(V)(\alpha\chi)^*=\Aut_{N_\alpha^\m{D}\beta}(U)$ and thus
$$\Aut_{N_\alpha^\m{D}\beta}(U)(\chi^{-1})^*=\Aut_{N_\alpha^\m{D}}(V)\alpha^*\leq \Aut_T(U).$$
Therefore $N_\alpha^\m{D}\beta\leq N_{\chi^{-1}}\cap T$, so $(\beta|_{N_\alpha^\m{D}})\eta\in \Hom_\m{D}(N_\alpha^\m{D},T)$ is well-defined and extends $\alpha$. This proves (\ref{ML6}).

\smallskip

We now derive the final contradiction. By (\ref{ML1}) and (\ref{ML6}), $\Aut_T(U)\not\in \Syl_p(\Aut_\m{D}(U))$. Let $\Aut_T(U)\leq S_U\in \Syl_p(\Aut_\m{D}(U))$. Then $\Aut_T(U)<\N_{S_U}(\Aut_T(U))$. Pick 
$$\alpha\in \N_{S_U}(\Aut_T(U))\backslash \Aut_T(U)$$ 
and note $P:=\N_T(U)=N_\alpha^\m{D}$. So by (\ref{ML6}), $\alpha$ extends to $\hat{\alpha}\in \Aut_\m{D}(P)$. Since $\alpha$ is a $p$-element, we may choose $\hat{\alpha}$ to be a $p$-element. Let $Q\in P^\m{D}\cap \m{D}^f$. By (\ref{ML3}), $|Q|>|U|$. Thus, it follows from (\ref{ML2}) that $\Aut_T(Q)\in \Syl_p(\Aut_\m{D}(Q))$. Hence, there exists $\beta\in \Hom_\m{D}(P,Q)$ such that $\hat{\alpha}\beta^*\in \Aut_T(Q)$. Pick $t\in \N_T(Q)$ such that $\hat{\alpha}\beta^*=c_t|_Q$. As $U\hat{\alpha}=U\alpha=U$, we have $(U\beta)^t=U\beta(\hat{\alpha}\beta^*)=U\hat{\alpha}\beta=U\beta$. Hence, $t\in \N_T(U\beta)$. As $U\in\m{D}^f$, it follows $P\beta=\N_T(U)\beta=\N_T(U\beta)$. Thus, $t\beta^{-1}\in P$ and $\hat{\alpha}=c_t|_Q\beta^{-*}=c_{t\beta^{-1}}|_P\in \Inn(P)$. This implies $\alpha=\hat{\alpha}|_U\in \Aut_T(U)$, contradicting the choice of $\alpha$ and thus completing the proof.
\end{proof}

\begin{lemma}\label{BCGLOAssHelp}
Let $Q\in\m{D}^c\backslash \m{H}$. Then there exists $\alpha\in\Hom_\m{D}(Q,T)$ such that $\C_{S_0}(Q_0\alpha)\not\leq Q_0\alpha$.
\end{lemma}

\begin{proof}
Assume the assertion is wrong and let $Q$ be a counterexample with $|Q_0|$ maximal. Since $Q\not\in \m{H}$, $Q_0\not\in\F_0^c$. In particular, $Q_0\neq S_0$, so $Q_0<R_0:=\N_{S_0}(Q_0)$. Set $R:=R_0Q$. Suppose first that $R_0\not\in\F_0^c$. Then $R\not\in \m{H}$. As $Q\in\m{D}^c$, we have also $R\in\m{D}^c$. Now the maximality of $|Q_0|$ yields that $R$ is not a counterexample and so there exists $\beta\in \Hom_\m{D}(R,T)$ such that $\C_{S_0}(R_0\beta)\not\leq R_0\beta$. As $Q_0\beta\leq R_0\beta$, it follows $\C_{S_0}(Q_0\beta)\not\leq Q_0\beta$ and the assertion holds with $\alpha=\beta|_Q$, contradicting $Q$ being a counterexample. So we have shown: 
\begin{equation}\label{BCGLO1}
R_0\in\F_0^c.
\end{equation}
We show next:
\begin{equation}\label{BCGLO2}
\mbox{There exists $\gamma\in\Hom_{\m{D}}(R,T)$ such that $R_0\gamma\in\F_0^f$ and $Q_0\gamma\in \N_{\F_0}(R_0\gamma)^f$.}
\end{equation}
Since $R_0\in\F_0^c$ by (\ref{BCGLO1}), it follows from \ref{wellplacedNormConj} that there exists $\eta\in \Hom_{\m{D}}(\N_T(R_0),T)$ such that $R_0\eta$ is well-placed. In particular, $R_0\eta\in\F_0^{fc}$, so we may choose $\m{N}\in\mathfrak{N}(R_0\eta)$. By \ref{ElPropN(P_0)}, $\m{N}$ is a saturated subsystem of $\m{D}$ on $\N_T(R_0\eta)$ with $\N_{\F_0}(R_0\eta)\unlhd\m{N}$ and $R_0\eta\unlhd \m{N}$. Hence, by \cite[Lemma~II.3.1]{Aschbacher/Kessar/Oliver:2011a}, there exists $\mu\in \Hom_{\m{N}}(\N_T(R_0\eta)\cap \N_T(Q_0\eta),\N_T(R_0\eta))$ such that $Q_0\eta\mu\in\m{N}^f$. Then $R_0\eta\mu=R_0\eta$ and, by \ref{PrelII} applied with $(\m{N},\N_{\F_0}(R_0\eta))$ in place of $(\F,\F_0)$, $Q_0\eta\mu\in \N_{\F_0}(R_0\eta)^f$. Observe that $R=R_0Q=\N_{S_0}(Q_0)Q\leq \N_T(Q_0)\cap \N_T(R_0)$ and thus $R\eta\leq \N_T(R_0\eta)\cap \N_T(Q_0\eta)$. Hence, (\ref{BCGLO2}) follows with $\gamma:=\eta\mu|_R$. 

\smallskip

Let now $\gamma$ be as in (\ref{BCGLO2}). As $\F_0$ is saturated, it follows from \cite[Lemma~II.3.1]{Aschbacher/Kessar/Oliver:2011a} that there exists $\delta\in\Hom_{\F_0}(\N_{S_0}(Q_0\gamma),S_0)$ such that $V_0:=Q_0\gamma\delta\in\F_0^f$. By \ref{PrelI}(a), $V_0\not\in\F_0^c$ and so, as $V_0\in\F_0^f$, $\C_{S_0}(V_0)\not\leq V_0$. If $\C_{S_0}(V_0)\not\leq R_0\gamma\delta$ then $\C_{S_0}(V_0)\cap \N_{S_0}(R_0\gamma\delta)\not\leq R_0\gamma\delta\geq V_0$, so $\C_{S_0}(V_0)\cap \N_{S_0}(R_0\gamma\delta)\not\leq V_0$. If $\C_{S_0}(V_0)\leq R_0\gamma\delta$ then also $\C_{S_0}(V_0)\leq \N_{S_0}(R_0\gamma\delta)$. So in any case,
$$\C_{S_0}(V_0)\cap \N_{S_0}(R_0\gamma\delta)\not\leq V_0.$$
By the choice of $\gamma$, $R_0\gamma\in\F_0^f$ and $Q_0\gamma\in \N_{\F_0}(R_0\gamma)^f$. Hence, by \cite[(2.2)(1)]{AschbacherGeneration}, 
$$(\N_{S_0}(Q_0\gamma)\cap \N_{S_0}(R_0\gamma))\delta=\N_{S_0}(V_0)\cap \N_{S_0}(R_0\gamma\delta).$$
Then $(\C_{S_0}(V_0)\cap \N_{S_0}(R_0\gamma\delta))\delta^{-1}\leq \C_{S_0}(Q_0\gamma)$ and $(\C_{S_0}(V_0)\cap \N_{S_0}(R_0\gamma\delta))\delta^{-1}\not\leq V_0\delta^{-1}=Q_0\gamma$. Hence $\C_{S_0}(Q_0\gamma)\not\leq Q_0\gamma$ and the assertion holds with $\alpha=\gamma|_Q$.
\end{proof}

\begin{lemma}\label{BCGLOAss}
Let $Q\in\m{D}^c\backslash \m{H}$. Then there exists $P\in Q^\m{D}$ such that 
$$\Aut_T(P)\cap O_p(\Aut_\m{D}(P))\not\leq \Inn(P).$$
\end{lemma}

\begin{proof}
By \ref{BCGLOAssHelp}, there exists $\alpha\in\Hom_\m{D}(Q,T)$ such that $\C_{S_0}(Q_0\alpha)\not\leq Q_0\alpha$. Then for $P:=Q\alpha$, $X:=\C_{S_0}(P_0)\not\leq P_0$. Note that $[P,N_X(P)]\leq P_0$ and $[P_0,X]=1$. So by \cite[8.2.2(b)]{KurzweilStellmacher2004}, $\Aut_X(P)\leq \C_{\Aut_\m{D}(P)}(P/P_0)\cap \C_{\Aut_\m{D}(P)}(P_0)\leq O_p(\Aut_\m{D}(P))$. If $\Aut_X(P)\leq \Inn(P)$ then, as $Q\in\m{D}^c$, $X\leq P$ and thus $X\leq P\cap S_0=P_0$, a contradiction. This proves the assertion.
\end{proof}

\begin{proposition}\label{DSat}
 $\m{D}$ is saturated. 
\end{proposition}

\begin{proof}
This follows from \cite[Thm.~2.2]{BCGLO2005} and the properties we have proved before: The set $\m{H}$ is closed under conjugation in $\m{D}$ according to \ref{PrelI}(a). By \ref{DGenG}, $\m{D}$ is $\m{H}$-generated. Since $\m{H}\subseteq\m{D}^c$, every subgroup in $\m{H}$ is fully $\m{D}$-centralized. Hence, by \ref{MainLemma}, $\m{D}$ is $\m{H}$-saturated. The assumption (*) in \cite[Thm.~2.2]{BCGLO2005} is verified in \ref{BCGLOAss}.
\end{proof}

\section{The proof of Theorem~\ref{MainThm}}\label{S6}

From the results we proved in previous sections, it remains to show that $\m{D}=\F_0T$ is the unique saturated subsystem $\m{E}$ of $\F$ on $T$ with $O^p(\m{E})=O^p(\F_0)$. We do so below in two lemmas. However, before we start, we want to recall that, for an arbitrary saturated fusion system $\F$ on $S$, 
$$\mathfrak{hyp}(\F)=\<[P,O^p(\Aut_\F(P))]:P\leq S\>$$
and $O^p(\F)$ is the fusion system on $\mathfrak{hyp}(\F)$ generated by the automorphisms groups $O^p(\Aut_\F(P))$ with $P\leq \mathfrak{hyp}(\F)$. See Section~I.7 in \cite{Aschbacher/Kessar/Oliver:2011a} for details, in particular for the proof that $O^p(\F)$ is a normal subsystem of $\F$. Observe also that $O^p(O^p(\F))=O^p(\F)$.

\begin{lemma}\label{OPD}
 $O^p(\F_0T)=O^p(\F_0)$.
\end{lemma}

\begin{proof}
Note $O^p(\F_0)\subseteq \F_0\subseteq \m{D}=\F_0T$. By \ref{Ac}, for any $P\leq T$, $O^p(\Aut_\m{D}(P))\leq\Ac(P)$, so $T_0:=\mathfrak{hyp}(\m{D})=\mathfrak{hyp}(\F_0)$. Moreover, by \ref{Lemma4}, $O^p(\Aut_\m{D}(P))=O^p(\Aut_{\F_0}(P))$ for $P\leq T_0\leq S_0$. Hence,
$$O^p(\m{D})=\<O^p(\Aut_\m{D}(P)):P\leq T_0\>_{T_0}=\<O^p(\Aut_{\F_0}(P)):P\leq T_0\>_{T_0}=O^p(\F_0).$$
\end{proof}

\begin{lemma}\label{OPE}
 If $\m{E}$ is a saturated subsystem of $\F$ on $T$ with $O^p(\m{E})=O^p(\F_0)$ then $\m{E}=\m{D}$.
\end{lemma}

\begin{proof}
Suppose the claim is true in the case $O^p(\F_0)=\F_0$. Then applying this property with $\m{E}=\m{D}$, we obtain $\m{D}=O^p(\F_0)T$, where we use \ref{OPD} and the fact that $\m{D}$ is saturated as proved in Section \ref{S5}. Hence, we are indeed reduced to the case that $O^p(\F_0)=\F_0$ and in particular, $\mathfrak{hyp}(\m{E})=\mathfrak{hyp}(\F_0)=S_0$. As $\F_0=O^p(\F_0)=O^p(\m{E})\unlhd \m{E}$, it follows from  \cite[7.18]{AschbacherGeneralized} that $P\cap S_0\in\F_0^c$ for any $P\in \m{E}^{frc}$. Moreover, $[P,O^p(\Aut_\m{E}(P))]\leq P\cap \mathfrak{hyp}(\m{E})=P\cap S_0$ and, for any $p^\prime$-element $\vphi\in \Aut_\m{E}(P)$, $\vphi|_{P_0}$ is a morphism in $O^p(\m{E})=\F_0$. Hence, $O^p(\Aut_\m{E}(P))\leq \Ac(P)$ and thus, by Alperin's Fusion Theorem \cite[Thm.~A.10]{BrotoLeviOliver2003},
$$\m{E}=\<O^p(\Aut_\m{E}(P)):P\in \m{E}^{frc}\>_T\subseteq \F_0T.$$
Alperin's Fusion Theorem together with \cite[7.18]{AschbacherGeneralized} and the fact that $\m{D}$ is saturated, gives also $\m{D}=\<O^p(\Aut_\m{D}(P)):P\in\m{D}^{fc},\;P_0\in\F_0^c\>_T$. So, by \ref{Ac}, 
it is sufficient to prove $\Ac(Q)\leq \Aut_\m{E}(Q)$ for $Q\in\m{D}^{fc}$ with $Q_0\in\F_0^c$. Let $\phi\in \Ac(Q)$ be a $p^\prime$-element. Then $\phi|_{Q_0}$ is a morphism in $\F_0$ and thus in $\m{E}$. As $\m{E}$ is saturated and $Q\in\m{D}^c$, it follows from \ref{littleLemma} and the extension axiom that $\phi|_{Q_0}$ extends to an element $\psi\in O^p(\Aut_{\m{E}}(Q))$. Then $[Q,\psi]\leq Q\cap \mathfrak{hyp}(\m{E})=Q_0$ and thus $\psi\in\Ac(Q)$. As $\m{D}$ is saturated and $Q\in\m{D}^f$, $\Aut_T(Q)\in \Syl_p(\Aut_\m{D}(Q))$. Hence, using \cite[8.2.2(b)]{KurzweilStellmacher2004}, we get $\phi\psi^{-1}\in \C_{\Ac(Q)}(Q_0)\leq O_p(\Ac(Q))\leq \Aut_T(Q)\leq \Aut_\m{E}(Q)$ and thus $\phi\in \Aut_\m{E}(Q)$. This proves the assertion.   
\end{proof}

\begin{proof}[{Proof of Theorem~\ref{MainThm}}]
 As proved in Section \ref{S5}, $\m{D}=\F_0T$ is saturated. By \ref{OPD} and \ref{OPE}, $\m{D}$ is the unique saturated subsystem $\m{E}$ of $\F$ on $T$ with $O^p(\m{E})=O^p(\F_0)$. Furthermore, \ref{Ac} gives $\Ac(P)=O^p(\Aut_\m{D}(P))$ for $P\leq T$ with $P_0\in\F_0^c$. This proves the theorem.
\end{proof}

\section{Final Remarks and Examples}\label{S7}

\subsection{Connections to factor systems}\label{factor}
We will explore here how the fusion system $\F_0T$ arises as a saturated preimage of certain subsystems of factor systems of $\F$.  As a basic fact, in a finite group $G$ with a normal subgroup $N$, for any subgroup $H$ of $G$, the product $NH$ is the largest preimage of the image of $H$ in $G/N$. We would like to establish similar properties of products in fusion systems.
Recall that, for any strongly closed subgroup $R$, the factor system $\F/R$ is defined; moreover, the strongly closed subgroups turn out to be precisely the kernels of morphisms between fusion systems; see e.g. \cite[Section~II.5]{Aschbacher/Kessar/Oliver:2011a} for the precise definition of $\F/R$ and more information. 
From now on, for any subsystem $\m{E}$ of $\F$ on a subgroup $E\leq S$, we write $\m{E}/R$ for the image of $\m{E}$ in $\F/R$, i.e. for the subsystem of $\F/R$ on $ER/R$ generated by the maps which are induced by morphisms from $\m{E}$. (With this notation we do not mean to imply in any way that $R$ is contained in $\m{E}$.) 
For a normal subsystem $\F_0$ of $\F$ on $S_0$, one defines the factor system $\F/\F_0$ to be $\F/S_0$. We set $\m{E}/\F_0:=\m{E}/S_0$. (Again, this notation doesn't mean that $S_0$ or $\F_0$ is contained in $\m{E}$.)

\smallskip

From the construction of $\F_0T$ it follows easily that $(\F_0T)/\F_0=\F_T(T)/\F_0$, so $\F_0T$ is a saturated preimage of $\F_T(T)/\F_0$. However, the following example shows that $\F_0T$ is neither the unique saturated preimage on $T$, nor the largest saturated preimage.

\begin{ex}\label{ExFactor1}
 Let $G_1$ and $G_2$ be two finite groups which both have a normal Sylow $p$-subgroup. Assume for at least one $i=1,2$, $G_i\neq O_p(G_i)\C_{G_i}(O_p(G_i))$. Set $G:=G_1\times G_2$ and let $T\in \Syl_p(G)$. Note that $T\unlhd G$ and thus, by \cite[Prop.~I.6.2]{Aschbacher/Kessar/Oliver:2011a}, $\F_0:=\F_T(T)\unlhd \F:=\F_T(G)$. Moreover, $\F/\F_0=\F_T(T)/\F_0$ is the fusion system on the trivial group. So $\F$ is the largest preimage of $\F_T(T)/\F_0$, but $\F_0T=\F_0$ is a proper subsystem of $\F$.
\end{ex}

We now turn to factor systems modulo an arbitrary strongly closed subgroup. Recall that, for any subgroup $R$ of $S$, we defined $\F_0R:=\F_0(RS_0)$.

\begin{proposition}
Let $R$ be a strongly closed subgroup (not necessarily containing $S_0$). Then $\F_0R/R=\F_0/R$.
\end{proposition}

\begin{proof}
 As $\F_0\subseteq \F_0R$, we have $\F_0/R\subseteq \F_0R/R$. Set $\ov{S}=S/R$ and $\ov{\F}=\F/R$. Accordingly, for any morphism $\alpha\in \F$, write $\ov{\alpha}$ for the image of $\alpha$ in $\ov{\F}$. Let $P\leq RS_0$ and $\vphi\in\Hom_{\F_0R}(P,S_0R)$. We need to show that $\ov{\vphi}$ is a morphism in $\ov{\F_0}=\F_0/R$. By Theorem~\ref{MainThm}, $\F_0R$ is saturated; so it follows from \cite[Thm.~II.5.9]{Aschbacher/Kessar/Oliver:2011a} that there exists $\psi\in\Hom_{\F_0R}(PR,S_0R)$ such that $\ov{\psi}=\ov{\vphi}$. Hence, replacing $(P,\vphi)$ by $(PR,\psi)$, we may assume $R\leq P$. Then $P=R(P\cap S_0)$ and so $\ov{P}=\ov{P\cap S_0}$. Moreover, by \ref{Decompose}, $\vphi_0:=\vphi|_{P\cap S_0}=c_r\phi$ for some $r\in R$ and $\phi\in\Hom_{\F_0}((P\cap S_0)^r,S_0)$. Hence, $\ov{\vphi}=\ov{\vphi_0}=\ov{\phi}\in \ov{\F_0}$ as required.
\end{proof}

Again, $\F_0R$ is not in any way unique or maximal as a saturated preimage of $\F_0/R$ on $S_0R$, as the following example shows.

\begin{ex}\label{ExFactor2}
 We continue to use the notation introduced in Example \ref{ExFactor1}. Take $R=T$ as a strongly closed subgroup. Then $\F_0/R=\F/R$ is the fusion system on the trivial group. However, as remarked before, $\F$ is the largest saturated preimage of $\F_0/R$ in $\F$, and $\F_0=\F_0T$ is a proper subsystem of $\F$. 
\end{ex}

\subsection{Products of $O^p(\F)$ with $p$-subgroups}
There is the following generalization of $O^p(\F)$ in the literature: For any $T\leq S$ which contains $\mathfrak{hyp}(\F)$, there is a saturated fusion subsystem
$$\F_T=\<O^p(\Aut_\F(P)):P\leq T\>_T$$
on $T$, which is normal in $\F$ provided $T\unlhd S$; see \cite[Thm.~I.7.4]{Aschbacher/Kessar/Oliver:2011a} for details. It is easy to see that $O^p(\F_T)=O^p(\F)$ and thus, by the uniqueness statement in Theorem~\ref{MainThm}, $\F_T=O^p(\F)T$. In particular, $O^p(\F)T$ is normal in $\F$ if $T\unlhd S$.

\subsection{Uniqueness of the Product} For the uniqueness statement in Theorem~\ref{MainThm} it is indeed essential to consider products \textit{inside the same fusion system} $\F$, as the following example shows:

\begin{ex}\label{7.4}
 We construct two saturated fusion systems $\F$ and $\m{G}$ on the same $p$-group such that $O^p(\F)=O^p(\m{G})$ and $\F\neq \m{G}$: 
Let $q\geq 3$ be a power of $p$, $1\neq\lambda\in GF(q)^\times$, and $S$ a finite dimensional vector space over $GF(q)$ of dimension at least $2$. Fix a non-trivial proper subspace $U$ of $S$ and complements $W_1,W_2$ of $U$ in $S$ with $W_1\neq W_2$. Define $\alpha_1,\alpha_2\in GL(S)$ via $\alpha_i|_U=\lambda\cdot \id_U$ and $\alpha_i|_{W_i}=\id_{W_i}$ for $i=1,2$. Set $G_i:=S\rtimes \<\alpha_i\>$ for $i=1,2$, $\F=\F_S(G_1)$ and $\m{G}=\F_S(G_2)$. Then for $\alpha:=\alpha_1|_U=\alpha_2|_U$,  $O^p(\F)=\F_U(U\rtimes\<\alpha\>)=O^p(\m{G})$. However, $\F\neq\m{G}$ as $W_1\neq W_2$. In particular, setting $\F_0:=O^p(\F)$, we have $(\F_0S)_\F\neq (\F_0S)_{\m{G}}$. 
\end{ex}

\subsection{The definition of $\F_0T$}

In our explicit description of $\F_0T$, one considers only the subgroups $P\leq T$ with $P\cap S_0\in\F_0^c$. This might seem a bit artificial on the first view. However, for an arbitrary subgroups $P\leq T$, it appears that there is no good way of describing $O^p(\Aut_{\F_0T}(P))$. In \ref{Ac} we prove that
$$O^p(\Aut_{\F_0T}(P))\leq \Ac(P),$$
but the converse inclusion does not necessarily hold, as we show in the next example.

\begin{ex}\label{7.5}
Let $p$ be a prime and $q\geq 3$ a power of $p$. Take $S$ to be a finite-dimensional vector space over $GF(q)$ which is the direct sum $S=U\oplus V\oplus W$ of three non-trivial subspaces $U,V,W$. Set $S_0:=U\oplus V$ and let $W^\prime\neq W$ be a complement of $V$ in $V\oplus W$. Let $\lambda\in GF(q)^\times$ and define $\alpha,\beta\in GL(S)$ via
\begin{eqnarray*}
 \alpha|_U&=&\lambda\cdot \id_U\mbox{ and }\alpha|_{V\oplus W}=\id_{V\oplus W},\\
 \beta|_{S_0}&=&\lambda\cdot \id_{S_0}\mbox{ and }\beta|_{W^\prime}=\id_{W^\prime}.
\end{eqnarray*}
Set
$$G:=S\rtimes \<\alpha,\beta\>\mbox{ and }N:=\<S_0,\beta\>.$$
Note that $S_0\unlhd G$, and that $\alpha$ and $\beta$ commute. 
Since $[S,\beta]=S_0$, this implies $N\unlhd G$. In particular, $\F_0:=\F_{S_0}(N)\unlhd \F:=\F_S(G)$. Set $P:=U\oplus W$. Then $P\cap S_0=U$, $[P,\alpha]=U$ and $\alpha|_U=\beta|_U\in\Aut_{\F_0}(U)$. Clearly, the order of $\alpha$ divides $q-1$, so $\alpha$ is a $p^\prime$-element. Hence,
$$\alpha|_P\in \Ac_{\F,\F_0}(P).$$
As $W\neq W^\prime$, no non-trivial element of $\<\beta\>$ normalizes $P$ and thus $\N_{NS}(P)=S$. Hence, $\Aut_{\F_0S}(P)=\Aut_{NS}(P)=1$ and, in particular, $\alpha|_P\not\in \Aut_{\F_0S}(P)$. This shows
$$\Ac_{\F,\F_0}(P)\not\leq \Aut_{\F_0S}(P).$$
\end{ex}

\bibliographystyle{amsalpha}
\bibliography{mybib}

\end{document}